\documentclass[american]{scrartcl}
\KOMAoptions{DIV=12, paper=a4, abstract=on}

\usepackage[utf8]{inputenc}
\usepackage[T1]{fontenc}
\usepackage{babel}
\usepackage{lmodern} 
\usepackage[babel=true]{microtype}

\usepackage{amsmath, amssymb, amsfonts, amsthm, booktabs, enumitem, subcaption}

\usepackage{siunitx}
\usepackage{graphicx, pgfplots}
\usetikzlibrary{calc, matrix}

\pgfplotsset{compat=newest,
   plot coordinates/math parser=false,
   scale only axis,
   xmajorgrids, xminorticks=false,
   ymajorgrids, yminorticks=false,
   every axis/.append style={%
      width=0.8\linewidth, height=0.45\linewidth,%
      font=\scriptsize, line width=0.75pt, mark = +,%
      scaled x ticks = false,
      legend cell align=left,%
   },
   legend style={at={(0.98,1.00)}, anchor=north east, align=left,
      fill=none,draw=none,row sep=-0.2em},
   every axis legend/.append style={font=\scriptsize},
}
\providecommand{\IhMark}{none}
\providecommand{\KminMark}{none}
\providecommand{\KminConvMark}{o}

\usepackage{hyperref}
\hypersetup{%
   pdfauthor = {W. Huang, L. Kamenski},
   pdftitle = {On the mesh nonsingularity of the moving mesh PDE method},
}
\usepackage[capitalize]{cleveref}
\crefformat{section}{#2Sect.~#1#3}
\Crefformat{section}{#2Section~#1#3}
\crefmultiformat{section}{Sects.~#2#1#3}%
   { and~#2#1#3}{, #2#1#3}{ and~#2#1#3}
\crefformat{equation}{(#2#1#3)}
\Crefformat{equation}{Equation~(#2#1#3)}
\crefmultiformat{equation}{(#2#1#3)}%
   { and~(#2#1#3)}{, (#2#1#3)}{ and~(#2#1#3)}
\crefrangeformat{equation}{(#3#1#4)--(#5#2#6)}

\providecommand{\Abs}[1]{\left\lvert#1\right\rvert}
\providecommand{\V}[1]{\boldsymbol{#1}}
\providecommand{\dx}{\,d\V{x}}
\providecommand{\dxi}{\,d\V{\xi}}
\providecommand{\p}[2]{\frac{\partial{}#1}{\partial{}#2}}
\providecommand{\Th}{\mathcal{T}_h}
\providecommand{\Tc}{\mathcal{T}_{c}}

\providecommand{\M}{\mathbb{M}}
\providecommand{\D}{\mathbb{D}}
\providecommand{\J}{\mathbb{J}}

\providecommand{\JMJ}{\J \M^{-1}\J^T}

\DeclareMathOperator{\tr}{tr}

\theoremstyle{plain}
\newtheorem{theorem}{\hspace{6mm}Theorem}[section]
\newtheorem{lemma}{\hspace{6mm}Lemma}[section]

\theoremstyle{definition}

\theoremstyle{remark}
\newtheorem{example}{\hspace{6mm}Example}[section]
\newtheorem{remark}{\hspace{6mm}Remark}[section]

\newcommand{\tcB}[1]{\textcolor{blue}{#1}}

\date{February 3, 2017}

\title{On the mesh nonsingularity\\
   of the moving mesh PDE method%
}

\author{%
   Weizhang Huang%
   \thanks{%
      The University of~Kansas, Department of~Mathematics, Lawrence, KS~66045, U.S.A.
      (\href{mailto:whuang@ku.edu}{\nolinkurl{whuang@ku.edu}}).%
      \newline{}
      Supported by the University of Kansas General Research Fund allocation~\#2301056.
   }%
   \and
   Lennard Kamenski%
   \thanks{%
      Weierstrass Institute for~Applied Analysis and~Stochastics, Berlin, Germany
      (\href{mailto:kamenski@wias-berlin.de}{\nolinkurl{kamenski@wias-berlin.de}}).%
   }%
}

\begin{document}
\maketitle

\begin{abstract}
The moving mesh PDE (MMPDE) method for variational mesh generation and adaptation
is studied theoretically at the discrete level, in particular the nonsingularity of the obtained meshes.
Meshing functionals are discretized geometrically and the MMPDE
is formulated as a modified gradient system of the corresponding discrete functionals
for the location of mesh vertices.
It is shown that if the meshing functional satisfies a coercivity condition,
then the mesh of the semi-discrete MMPDE is nonsingular for all time if it is nonsingular initially.
Moreover, the altitudes and volumes of its elements are bounded below by positive numbers depending
only on the number of elements, the metric tensor, and the initial mesh.
Furthermore, the value of the discrete meshing functional is convergent as time increases,
which can be used as a stopping criterion in computation.
Finally, the mesh trajectory has limiting meshes which are critical points of the discrete functional.
The convergence of the mesh trajectory can be guaranteed when a stronger condition
is placed on the meshing functional.
Two meshing functionals based on alignment and equidistribution are known to satisfy the coercivity condition.
The results also hold for fully discrete systems of the MMPDE provided that
the time step is sufficiently small and a numerical scheme preserving the property of monotonically decreasing
energy is used for the temporal discretization of the semi-discrete MMPDE.\@
Numerical examples are presented.
\\[2ex]
\noindent\textbf{AMS 2010 MSC.}
65N50, 65K10
\\[1ex]
\end{abstract}


\section{Introduction}
The variational method for mesh generation and adaptation has received considerable attention in the scientific computing community; e.g.,\ see~\cite{Car97,HuaRus11,KnuSte94,Lis99,ThoWarMas85} and references therein.
It generates an adaptive mesh as the image of a given reference mesh under a coordinate transformation determined by a meshing functional.
Such a functional is typically designed to measure difficulties in the numerical approximation of the physical solution and involve a user-prescribed metric tensor (monitor function) to control mesh adaptation.
This method has the advantage that it makes it easy to incorporate mesh requirements (e.g.,\ smoothness, adaptivity, or alignment) in the formulation of the functional~\cite{BraSal82}.
It serves as not only a standalone method for mesh generation and adaptation but also a smoothing device for automatic mesh generation (e.g.,\ see~\cite{DasKamSi16,FreOll97,HuaKamSi15}).
Moreover, the variational method is the base for a number of adaptive moving mesh methods~\cite{HuaRenRus94a,HuaRus99,HuaRus11,LiTanZha01}.

A number of variational methods have been developed so far.
For example, Winslow~\cite{Win81} proposes an equipotential method based on variable diffusion.
Brackbill and Saltzman~\cite{BraSal82} develop a method by combining mesh concentration,
smoothness, and orthogonality.
Dvinsky~\cite{Dvi91} uses the energy of harmonic mappings as his meshing functional.
Knupp~\cite{Knu96} and Knupp and Robidoux~\cite{KnuRob00} formulate functionals based
on the idea of conditioning the Jacobian matrix of the coordinate transformation. 
Huang~\cite{Hua01a} and Huang and Russell~\cite{HuaRus11} develop functionals based
on the so-called equidistribution and alignment conditions.

Compared with the algorithmic development, much less progress has been made
in theoretical studies.
The existence and uniqueness of the minimizer of Dvinsky's meshing functional is guaranteed
by the theory of harmonic mappings between multidimensional  domains~\cite{Dvi91}.
Winslow's functional~\cite{Win81} is uniformly convex and coercive so it has a unique minimizer~\cite[Example 6.2.1]{HuaRus11}.
Huang's functional~\cite{Hua01a} is coercive and polyconvex and has minimizers~\cite[Example 6.2.2]{HuaRus11} while that of Huang and Russell is coercive and polyconvex and has a nonsingular minimizer~\cite[Example 6.2.3]{HuaRus11}.
Note that the nonsingularity of the minimizer for the above mentioned functionals is unknown 
(except for that of Huang and Russell).
Moreover, these results are only at the continuous level.

At the discrete level, studies mainly remain in one spatial dimension.
Pryce~\cite{Pry89} proves the existence of the limiting mesh and the convergence of
de Boor's algorithm for solving the equidistribution principle when the metric tensor is approximated by a piecewise linear interpolant.
His result is generalized by Xu et al.~\cite{XuHuaRus11} to the situation where the metric tensor is approximated by a piecewise constant interpolant.
Gander and Haynes~\cite{GanHay12} and  Haynes and Kwok~\cite{HayKwo17} show the existence of the limiting mesh and the convergence of  the parallel and alternating Schwarz domain decomposition algorithms applied to the continuous and discrete equidistribution principle.

The objective of this paper is to present a theoretical study on variational mesh generation and
adaptation at the discrete level for any dimension. We consider a broad class
of meshing functionals and the MMPDE method for finding their minimizers. We employ
a geometric discretization recently introduced in~\cite{HuaKam15a} for meshing functionals.
The semi-discrete MMPDE (discrete in space and continuous in time) is defined as
a modified gradient system for the corresponding discrete functionals.
The mesh nonsingularity and the existence and uniqueness of limiting meshes for
both the semi-discrete and fully discrete MMPDEs are studied.
Largely thanks to the inherent properties of the new discretization,
it can be shown that if the meshing functional satisfies a coercivity condition,
the mesh of the semi-discrete MMPDE stays nonsingular for all time if it is nonsingular
initially. Moreover, the altitudes and volumes of its elements are bounded from
below by positive numbers depending only on the number of elements, the metric tensor, and the initial mesh.
Furthermore, the value of the discrete functional
is convergent as time increases, which can be used as a stopping criterion for the computation.
Finally, the mesh trajectory has limiting meshes which are critical points of the discrete functional.
The convergence of the mesh trajectory can be guaranteed when a stronger condition
is placed on the meshing functional (see the discussion following \cref{thm-lim-2}).
The functionals based on alignment and equidistribution~\cite{Hua01a,HuaRus11}
are known to satisfy the coercivity condition for a large range of parameters
(see \cref{huang} with $p > 1$).
The analysis also holds for a fully discrete system for the MMPDE provided that
the time step is sufficiently small and a numerical scheme preserving the property of monotonically decreasing
energy is used to integrate the semi-discrete MMPDE.\@
Euler, backward Euler, and algebraically stable Runge-Kutta schemes 
(including Gauss and Radau IIA) are known to preserve the property
under a time-step restriction that involves a local Lipschitz bound of the Hessian matrix
of the discrete functional (e.g.,~\cite{HaiLub14,StuHum96}).
 
An outline of this paper is given as follows. Meshing functionals
and the MMPDE method for finding minimizers are described in \cref{SEC:fun}.
The geometric discretization of the meshing functionals is given in \cref{SEC:dis}.
\Cref{SEC:analysis} is devoted to the analysis of the semi-discrete MMPDE and its discretization.
Numerical examples are given in \cref{SEC:numerics} to demonstrate the theoretical findings.
The conclusions are drawn in \cref{SEC:conclusion}.

\section{Meshing functionals and MMPDE}
\label{SEC:fun}

In this section we describe the general form of a functional and two specific examples
used for mesh generation and adaptation. We also discuss the concept of functional equivalence
and the MMPDE approach for finding a minimizer of the meshing functional.

Let $\Omega$ be a bounded, not necessarily convex, polygonal or polyhedral domain in $\mathbb{R}^d$, $d \ge 1$, and
$\M = \M(\V{x})$ a symmetric metric tensor defined on $\Omega$ which satisfies
\begin{equation}
\label{M-1}
\underline{m}\ I \le \M(\V{x}) \le \overline{m}\ I \quad \forall \V{x} \in \Omega,
\end{equation}
where the inequality sign is in the sense of positive definiteness and $\underline{m},\overline{m} >0$ are constants.
Our goal is to use the variational method to generate a simplicial mesh for $\Omega$ according to $\M$.

Let  $\Omega_c$ be a chosen computational domain, which can be a real domain in
$\mathbb{R}^d$ or a collection of simplexes (see the discussion in the next section).
With the variational method, an adaptive mesh is generated as the image
of a computational mesh on $\Omega_c$ under a coordinate transformation between $\Omega_c$ and $\Omega$
which in turn is determined as a minimizer of a meshing functional.

Denote the coordinate transformation by $\V{x} = \V{x}(\V{\xi}): \Omega_c \to \Omega$ and its inverse coordinate
transformation by $\V{\xi} = \V{\xi}(\V{x}): \Omega \to \Omega_c$. We consider meshing functionals in the form
\begin{equation}
I[\V{\xi}] = \int_\Omega G \left (\J, \det(\J), \M, \V{x} \right) \dx,
\label{fun-1}
\end{equation}
where $\J = \frac{\partial \V{\xi}}{\partial \V{x}}$ is the Jacobian matrix of $\V{\xi} = \V{\xi}(\V{x})$ and
$G$ is a given function. We assume that $G$ has continuous derivatives up to the third order 
with respect to all of its arguments, ${\|\J\|} < \infty$, $|\det(\J)| < \infty$, $\M$ is symmetric
and uniformly positive definite on $\Omega$, and $\V{x} \in \Omega$.
($\| \cdot \|$ denotes the matrix/vector 2-norm.)
This functional is minimized for the coordinate transformation subject to suitable boundary correspondence
between $\partial \Omega$ and $\partial \Omega_c$.
The form \cref{fun-1} is very general and includes many existing meshing functionals as special examples
(e.g.,~\cite{HuaRus11,KnuSte94,Lis99} and \cref{ex:Winslow,ex:Huang} below).

The functional \cref{fun-1} is formulated in terms of the inverse coordinate transformation $\V{\xi}(\V{x})$.
It can be transformed into a mathematically equivalent functional expressed in terms of
the coordinate transformation $\V{x}(\V{\xi})$.
To explain this, we consider a coordinate transformation
$(\V{x}, \V{\xi}) \to (\V{u},\V{v})$ defined by
\begin{equation}
\label{uv-1}
\V{x} = \Phi(\V{u}, \V{v}),\quad \V{\xi} = \Psi(\V{u}, \V{v}),\qquad
\det \left (\frac{\partial (\V{x}, \V{\xi})}{\partial (\V{u}, \V{v})} \right ) \neq 0,
\end{equation}
where $\V{u}$ and $\V{v}$ are the new independent and dependent variables, respectively.
The curve given by the equation $\V{\xi} = \V{\xi}(\V{x})$
in the $\V{x}$-$\V{\xi}$ space corresponds to the curve given by some equation
$\V{v} = \V{v}(\V{u})$ in the $\V{u}$-$\V{v}$ space. Making the change of variables \cref{uv-1},
we can transform the functional \cref{fun-1} into a new functional involving $\V{u}$ and $\V{v}$.
The invariance of the Euler-Lagrange equation in calculus of variations (e.g.,~Gelfand and Fomin~\cite{GelFom63})
states that if $\V{\xi} = \V{\xi}(\V{x})$ satisfies the Euler-Lagrange equation of \cref{fun-1},
then $\V{v} = \V{v}(\V{u})$ satisfies the Euler-Lagrange equation of the new functional.
Thus, the minimizers of \cref{fun-1} can be obtained through the minimizers of
the new functional, and vice versa. In this sense, we say \cref{fun-1} and the new functional
are mathematically equivalent.

Consider a special coordinate transformation
\begin{equation}
\label{uv-2}
\V{x} = \Phi(\V{u}, \V{v}) \equiv \V{v},\quad \V{\xi} = \Psi(\V{u}, \V{v}) \equiv \V{u} ,
\end{equation}
which represents an interchange of the roles of the independent and dependent variables.
Since the Jacobian matrix of \cref{uv-2} is
\[
\frac{\partial (\V{x}, \V{\xi})}{\partial (\V{u}, \V{v})}
= \begin{bmatrix} 0 & I \\ I & 0 \end{bmatrix} ,
\]
which is nonsingular, the invariance of the Euler-Lagrange equation implies that
the functional \cref{fun-1} is mathematically equivalent to
\begin{equation}
I[\V{x}] = \int_{\Omega_c} \frac{G \left (\J, \det(\J), \M, \V{x} \right)}{\det(\J)} \dxi ,
\label{fun-2}
\end{equation}
which is obtained by interchanging the roles of its independent
and dependent variables in \cref{fun-1}. Notice that the new functional is still denoted by
$I$ without causing confusion. Indeed, from the equivalence, we can consider \cref{fun-1}
as a functional for $\V{\xi} = \V{\xi}(\V{x})$ (\emph{the $\V{\xi}$-formulation}) or for 
$\V{x} = \V{x}(\V{\xi})$ through the interchanging the roles of the independent and dependent
variations, i.e., \cref{fun-2} (\emph{the $\V{x}$-formulation}).

In this work, we use the $\V{x}$-formulation. We employ the MMPDE method (a time-transient approach~\cite{HuaRenRus94,HuaRenRus94a}) to find a minimizer of the functional \cref{fun-2}.
The MMPDE is defined as a modified gradient flow of \cref{fun-2},
\begin{equation}
\label{mmpde-2}
\p{\V{x}}{t} = - \frac{P}{\tau} \frac{\delta I}{\delta \V{x}}, \quad t > 0, 
\end{equation}
where $\frac{\delta I}{\delta \V{x}}$ is the functional derivative of $I$ with respect to $\V{x}$,
$\tau > 0$ is a constant parameter used to adjust the time scale of the equation, and $P$ is a positive scalar
function used to make the equation to have some invariance properties (a choice of $P$ will be given later for \cref{ex:Winslow,ex:Huang}).
A discretization of \cref{mmpde-2} gives a system for the nodal velocities for the physical mesh.
The interested reader is referred to~\cite{HuaRus11} for detailed discussion on the discretization of MMPDEs.
In the next section, we consider a direct discretization method with which the functional \cref{fun-1}
(instead of MMPDEs) is discretized directly and the nodal velocity system is
then obtained as a modified gradient system of the discretized functional.


\begin{example}[The generalized Winslow functional]
\label{ex:Winslow}
The first example is a generalization of Winslow's variable diffusion functional~\cite{Win81},
\begin{equation}
   I[\V{\xi}] = \int_\Omega \tr(\JMJ) \dx,
   \label{winslow}
\end{equation}
where $\tr(\cdot)$ denotes the trace of a matrix and $\M^{-1}$ serves as the diffusion matrix.
This functional has been used by many researchers; e.g.,~\cite{BecMacRob01,HuaRus98,HuaRus99,LiTanZha01}.
It is coercive and convex (in terms of $\V{\xi}=\V{\xi}(\V{x})$)
and therefore has a unique minimizer~\cite[Example 6.2.1]{HuaRus11}.

For the discretization to be discussed in the next section, we need 
the derivatives of $G$ with respect to $\J$, $\det(\J)$, $\M$, and $\V{x}$. They are
\begin{equation}
   \begin{cases}
     G  = \tr(\JMJ),\\
      \frac{\partial G}{\partial \J}  =  2 \M^{-1} \J^T,\\
      \frac{\partial G}{\partial \det(\J)}  =  0 ,\\
      \frac{\partial G}{\partial \M}  =  - \M^{-1} \J^T \J \M^{-1},\\
      \frac{\partial G}{\partial \V{x}}  =   0.
   \end{cases}
   \label{winslow-2}
\end{equation}
Note that $\frac{\partial G}{\partial \J}$ and $\frac{\partial G}{\partial \M}$ are $d$-by-$d$ matrices and
$\frac{\partial G}{\partial \V{x}}$ is a row vector of $d$ components. They are expressed
in the notation of scalar-by-matrix differentiation. For example, $\frac{\partial G}{\partial \J}$
is a $d$-by-$d$ matrix defined as
\begin{equation}
   {\left (\frac{\partial G}{\partial \J}\right )}_{ij} = \frac{\partial G}{\partial \J_{ji}} .
\label{sbmd-1}
\end{equation}
The chain rule for scalar-by-matrix differentiation reads as%
\footnote{The interested reader is referred to~\cite{HuaKam15a} for a more detailed
discussion on scalar-by-matrix differentiation.}
\begin{equation}
\frac{\partial G}{\partial t} = \tr\left ( \frac{\partial G}{\partial \J} \frac{\partial \J}{\partial t}\right ).
\label{sbmd-2}
\end{equation}
Using the definition of $G$ and viewing $\J$ as a function of $t$, we have
\[
\frac{\partial G}{\partial t} = \tr\left ( \frac{\partial \JMJ}{\partial t}\right )
= \tr\left ( \frac{\partial \J}{\partial t}\M^{-1} \J^T + \J \M^{-1} \frac{\partial \J^T}{\partial t}\right )
= \tr\left ( 2 \M^{-1} \J^T \frac{\partial \J}{\partial t} \right ) .
\]
By comparing this with the chain rule, we get
\[
\frac{\partial G}{\partial \J}  =  2 \M^{-1} \J^T .
\]
The other derivatives are obtained similarly.

Regarding the choice of $P$, it is useful to make the MMPDE invariant under
the scaling transformation $\M \to c \cdot \M$ for a positive number $c$ since
the mesh concentration is controlled by the distribution of $\M$ instead of its absolute value.
A choice of $P$ for this purpose for the current functional is
\[
   P = {\det(\M)}^{\frac{1}{d}} .
\]
\qed{}
\end{example}

\begin{example}[Huang's functional]
\label{ex:Huang}
The second functional is
\begin{equation}
   I[\V{\xi}]
   = \theta \int_\Omega \sqrt{\det(\M)} {\left(\tr(\JMJ) \right)}^{\frac{dp}{2}} \dx
   + (1 - 2 \theta) d^{\frac{dp}{2}} \int_\Omega \sqrt{\det(\M)}
            {\left( \frac{\det(\J)}{\sqrt{\det(\M)}}\right)}^{p} \dx,
   \label{huang}
\end{equation}
where $0 \le \theta \le 1$ and $p > 0$ are dimensionless parameters.
This functional was proposed by Huang~\cite{Hua01a} based on the so-called alignment and equidistribution conditions.
Alignment and equidistribution are balanced by $\theta$,
with full alignment for $\theta = 1$ and full equidistribution for $\theta = 0$.
For $0 < \theta \le \frac{1}{2}$, $d p \ge 2$, and $p \ge 1$, the functional is coercive and polyconvex
(in terms of $\V{\xi}=\V{\xi}(\V{x})$) and has a minimizer~\cite[Example 6.2.2]{HuaRus11}.
Moreover, for $\theta = \frac{1}{2}$ it reduces to
\begin{equation}
\label{huang-2}
   I[\V{\xi}] = \frac{1}{2} \int_\Omega \sqrt{\det(\M)} {\left (\tr(\JMJ)\right )}^{\frac{dp}{2}} \dx,
\end{equation}
which is coercive and convex (in terms of $\V{\xi}=\V{\xi}(\V{x})$) and has a unique minimizer.
Particularly, \cref{huang-2} gives the energy functional for a harmonic mapping
from $\Omega$ to $\Omega_c$ when $dp/2 = 1$ (cf.~\cite{Dvi91}).
Moreover, \cref{huang-2} and Winslow's functional \cref{winslow} coinside when
$dp/2 = 1$ and $\M = I$.

The derivatives of $G$ are
\begin{equation}
   \begin{cases}
      G = \theta \sqrt{\det(\M)} {\left(\tr(\JMJ) \right)}^{\frac{dp}{2}}
      + (1 - 2 \theta) d^{\frac{dp}{2}} \sqrt{\det(\M)} {\left( \frac{\det(\J)}{\sqrt{\det(\M)}}\right)}^{p},\\
      \p{G}{\J} =  d p \theta \sqrt{\det(\M)}  {\left( \tr(\JMJ )\right )}^{\frac{d p}{2}-1} \M^{-1} \J^T,\\
      \p{G}{r} = p (1-2\theta) d^{\frac{d p}{2}} \det{(\M)}^{\frac{1-p}{2}} \det{(\J)}^{p-1},\\
      \p{G}{\M} = - \frac{\theta d p}{2} \sqrt{\det(\M)} {\left(\tr(\JMJ) \right)}^{\frac{dp}{2}-1} \M^{-1} \J^T \J \M^{-1}
      + \frac{\theta}{2} \sqrt{\det(\M)} {\left(\tr(\JMJ) \right)}^{\frac{dp}{2}} \M^{-1} \\
      \qquad + \frac{(1-2\theta)(1-p)d^{\frac{d p}{2}}}{2} \sqrt{\det(\M)} {\left( \frac{\det(\J)}{\sqrt{\det(\M)}}\right)}^{p} \M^{-1} ,\\
      \p{G}{\V{x}} = 0.
   \end{cases}
   \label{huang-3}
\end{equation}

A choice of $P$ to make the MMPDE invariant under the scaling transformations of $\M$
for the current functional is
\[
   P = {\det(\M)}^{\frac{p-1}{2}} .
\]
\qed{}
\end{example}

\section{A geometric discretization of meshing functionals}
\label{SEC:dis}

Let  $\Th = \{ K\}$ be the target simplicial mesh on $\Omega$ and $N$ and $N_v$ the numbers
of its elements and vertices, respectively.
We assume that the computational mesh $\Tc = \{ K_c\}$ is chosen to satisfy the following properties:
\begin{enumerate}[label=(\alph*)]

\item It has the same $N$ and $N_v$ as $\Th$.

\item There is a one-to-one correspondence between the elements of $\Tc$ and those of $\Th$.

\item $\Tc$ has the same connectivity when $\Tc$ is a real mesh (see the explanation below).

\item
There exist $\underline{\rho}, \overline{\rho} > 0$ such that
\begin{equation}
   \underline{\rho} N^{-\frac 1 d} \le \rho_{\!_{K_c}}
   \quad \text{and} \quad h_{K_c} \le \overline{\rho} N^{-\frac 1 d}
   \qquad \forall K_c \in \Tc
   ,
   \label{c-mesh-1}
\end{equation}
where $h_{K_c}$ and $\rho_{\!_{K_c}}$ denote the diameter and in-diameter (the diameter of the largest inscribed ellipsoid) of $K_c$, respectively.
Note that \cref{c-mesh-1} implies the conventional mesh regularity condition 
\[
   \frac{h_{K_c}}{\rho_{\!_{K_c}}} \le \frac{\overline{\rho}}{\underline{\rho}}
   \qquad \forall K_c \in \Tc
   .
\]
Moreover, it implies $|K_c| = \mathcal{O}(1/N)$ for all $K_c \in \Tc$, where $|K_c|$ denotes the volume of $K_c$.
\end{enumerate}

$\Tc$ can be a real mesh or a collection of $N$ simplexes.
For example, $\Tc$ can be a simplicial mesh induced from a rectangular/cubic mesh when $\Omega$ has a simple geometry.
For more complicated geometries, it can be a Delaunay mesh of $N_v$ uniformly distributed points.
For meshing functionals that are invariant under rotations and translations of the $\V{\xi}$-coordinates (such as the functionals considered in the previous section), $\Tc$ can be a collection of $N$ copies of the ``master'' element ${N^{-\frac{1}{d}}} \hat{K}$, where $\hat{K}$ is a given unitary equilateral simplex and ${N^{-\frac{1}{d}}} \hat{K}$ denotes the simplex resulting from scaling $\hat K$ by a factor of ${N^{-\frac{1}{d}}}$.
To see this, imagine that $\Tc$ is a uniform mesh so that each of its elements can be transformed to the master element using rotation and translation.
Then, due to the invariance property, any elment-wise approximation (see the discussion of the discretization below) of such a meshing functional using the affine mapping between $K \in \Th$ and its counterpart $K_c \in \Tc$ is unchanged if the affine mapping is replaced by that between $K$ and ${N^{-\frac{1}{d}}} \hat{K}$.
As a result, the discretization of the functional can be regarded as being carried out between $\Th$ and ${N^{-\frac{1}{d}}} \hat{K}$ or the collection of $N$ copies of ${N^{-\frac{1}{d}}} \hat{K}$.

The advantage of using the collection of $N$ copies of a simplex is obvious: no real computational mesh is needed anymore.
This is especially convenient if $\Omega$ has a complicated geometry for which a mesh with reasonable quality is difficult to obtain or for applications where it is burdensome to define a computational mesh (such as mesh smoothing, see \cref{ex:cami1a}).
On the other hand, a real computational mesh allows elements of different size and shape, which can be desirable in some applications.

For any element $K \in \Th$ and the corresponding element $K_c \in \Tc$, let $F_K: K_c \to K$ be
the affine mapping between them and $F_K'$ its Jacobian matrix.
Let the vertices of $K$ and $K_c$ be $\V{x}_j^K$, $j=0, \dotsc, d$ and $\V{\xi}_j^K$, $j=0, \dotsc, d$, respectively. 
It holds
\begin{equation}
\label{FK-1}
F_K' = E_K \hat{E}_K^{-1},\quad {(F_K')}^{-1} = \hat{E}_K E_K^{-1},\quad
|K| = \frac{1}{d!}|\det(E_K)|, \quad |K_c| = \frac{1}{d!}|\det(\hat{E}_K)|,
\end{equation}
where the edge matrices $E_K$ and $\hat{E}_K$ are defined as
\[
E_K = [\V{x}_1^K-\V{x}_0^K, \dotsc, \V{x}_d^K - \V{x}_0^K],\quad
\hat{E}_K = [\V{\xi}_1^K-\V{\xi}_0^K, \dotsc, \V{\xi}_d^K - \V{\xi}_0^K] .
\]
Let
\[
\Omega_c = \bigcup_{K_c \in \Tc} K_c .
\]
Note that $\Omega_c$ may be a real domain in $\mathbb{R}^d$ or a collection of $N$ simplexes.

We now describe the geometric discretization~\cite{HuaKam15a} for the functional \cref{fun-1}.
The idea is simple: the coordinate transformation $\V{x} = \V{x}(\V{\xi})$ is approximated
by the piecewise linear mapping $\{ F_K, \; K \in \Th\}$ and the integral in \cref{fun-1}
is approximated by the midpoint quadrature rule.\footnote{A more accurate quadrature rule
could be used; however, our numerical experience shows that
the simple midpoint quadrature rule works well for problems tested.}
This results in a Riemann sum which can be considered as a function
of the location of the physical vertices (in the $\V{x}$-formulation), according to the functional equivalence
discussed in the preceding section. (Note that $\Tc$ is given and, thus, known.)
From $\J \approx {(F_K')}^{-1} = \hat{E}_K E_K^{-1}$ on $K$, we have
\begin{equation}
\label{fun-3}
I_h(\V{x}_1, \dotsc, \V{x}_{N_v}) = \sum_{K \in \Th} |K|
\, G(\hat{E}_K E_K^{-1}, \frac{\det(\hat{E}_K)}{\det(E_K)},\M_K, \V{x}_K),
\end{equation}
where $\V{x}_K$ is the centroid of $K$ and $\M_K = \frac{1}{d+1} \sum_{i=0}^d \M(\V{x}_i^K)$.
As for the continuous case, the MMPDE for \cref{fun-3} is defined as
\begin{equation}
\label{mmpde-3}
\frac{d \V{x}_i}{d t} = - \frac{P(\V{x}_i)}{\tau} {\left (\frac{\partial I_h}{\partial \V{x}_i}\right )}^T,
\quad i = 1, \dotsc, N_v, \; t > 0 .
\end{equation}
The derivatives on the right-hand side of the mesh equation \cref{mmpde-3} can be found analytically
in a compact matrix form (see~\cite{HuaKam15a} for the derivation):
\begin{equation}
\label{mmpde-4}
\frac{d \V{x}_i}{d t} = \frac{P(\V{x}_i)}{\tau} \sum_{K \in \omega_i} |K| \V{v}_{i_K}^K , \quad i = 1, \dotsc, N_v,
\end{equation}
where $\omega_i$ is the patch of the elements having $\V{x}_i$ as one of their vertices
and $i_K$ and $\V{v}_{i_K}^K$ are the local index and velocity of vertex $\V{x}_i$ on the element $K$,
respectively.  The local velocities are 
\begin{align}
\begin{bmatrix} {(\V{v}_{1}^K)}^T \\ \vdots \\ {(\V{v}_{d}^K)}^T \end{bmatrix}
 =& - G  E_K^{-1}
      +  E_K^{-1} \p{G}{\J} \hat{E}_K E_K^{-1}
      +  \p{G}{\det(\J)} \frac{\det(\hat{E}_K)}{\det(E_K)} E_K^{-1}
\label{mmpde-5} \\
 &\qquad \quad - \frac{1}{d+1} \sum_{j=0}^{d} \tr\left( \p{G}{\M} \M_{j, K} \right)
      \begin{bmatrix} \p{\phi_{j,K}}{\V{x}} \\ \vdots \\ \p{\phi_{j,K}}{\V{x}} \end{bmatrix} 
     - \frac{1}{d+1} \begin{bmatrix} \p{G}{\V{x}} \\ \vdots \\ \p{G}{\V{x}} \end{bmatrix},
\notag \\
{(\V{v}_{0}^K)}^T =& - \sum_{k=1}^d {(\V{v}_{k}^K)}^T 
- \sum_{j=0}^{d} \tr\left( \p{G}{\M} \M_{j, K} \right) \p{\phi_{j,K}}{\V{x}}
-  \p{G}{\V{x}} ,
\label{mmpde-6}
\end{align}
where $\M_{j,K}=\M(\V{x}_j^K)$, $\phi_{j,K}$ is the linear basis function associated with $\V{x}_j^K$, and
\[
G, \quad \frac{\partial G}{\partial \J}, \quad \frac{\partial G}{\partial \det(\J)},
\quad \frac{\partial G}{\partial \M}, \quad \text{and} \quad \frac{\partial G}{\partial \V{x}}
\]
are evaluated at
\[
   \J = \hat{E}_K E_K^{-1}, \quad \det(\J) = \frac{\det(\hat{E}_K)}{\det(E_K)},
   \quad \M = \M_K,
   \quad \V{x} = \V{x}_K.
\]

The MMPDE~\cref{mmpde-3} should be modified properly for boundary vertices:
if $\V{x}_i$ is a fixed boundary vertex, the corresponding equation is replaced by
\[
   \frac{d \V{x}_i}{d t} = 0,
\]
and when $\V{x}_i$ is allowed to move on a boundary curve (in 2D) or surface (in 3D) represented by
$
   \phi(\V{x}) = 0
$,
then the mesh velocity $\p{\V{x}_i}{t}$ needs to be modified such that its normal component along the curve or surface is zero, i.e.,
\[
   \nabla \phi (\V{x}_i) \cdot \frac{d \V{x}_i}{d t} = 0.
\]

\begin{remark}
\label{rem-mmpde-3}
The formulation of the MMPDE \cref{mmpde-3} is similar to that of a spring model
for mesh movement (cf.~\cite[Section 7.3.2]{HuaRus11}),
with the right-hand side term acting as the sum of the spring forces between $\V{x}_i$ and its neighboring vertices.
This makes it amenable to time integration by both explicit and implicit schemes.
On the other hand, \cref{mmpde-3} is different from existing spring models for mesh movement.
It does not involve parameters such as spring constants that typically need fine tuning.
Moreover, the forces in \cref{mmpde-3} are defined based on the global meshing
functional \cref{fun-1}.
This property is very important since it provides a good chance to prevent the mesh from
becoming singular.
For example,  for the functional \cref{huang} the forces are defined to keep the mesh elements as regular
and uniform in the metric $\M$ as possible.
\qed{}
\end{remark}

\section{Mesh nonsingularity and existence of the limiting meshes}
\label{SEC:analysis}

In this section we study the nonsingularity of the mesh trajectory and the existence of the limiting meshes
as $t \to \infty$ for the semi-discrete MMPDE \cref{mmpde-3} and its discretization.

\subsection{Two lemmas}

The functionals in \cref{ex:Winslow,ex:Huang} involve a factor
$\tr(\J \M^{-1} \J^T) = \tr ({(F_K')}^{-1}\M_K^{-1} {(F_K')}^{-T})$. An equivalent form of it is
$\| {(F_K')}^{-1}\M_K^{-1} {(F_K')}^{-T}\|$. We first obtain a geometric interpretation for it, which is needed later
in our analysis.

\begin{lemma}
\label{lem-FK-2}
Let $\tilde{K}$ be an equilateral simplex and $K$ an arbitrary simplex in $\mathbb{R}^d$, $F_K: \tilde K \to K$ the affine mapping between them, and $\M_K$ a constant symmetric and positive definite matrix. Then,
\begin{align}
\frac{\tilde{a}^{2}}{a_{K,\M}^2} \le 
\| {(F_K')}^{-1}\M_K^{-1} {(F_K')}^{-T}\| \le \frac{d^2 \, \tilde{a}^2}{a_{K, \M}^2} ,
\label{lem-FK-2-1}
\end{align}
where $a_{K,\M}$ is the minimum altitude of $K$ in the metric $\M_K$ and $\tilde{a}$
is the altitude of $\tilde{K}$.
\end{lemma}

\begin{proof}

Let $\tilde \phi_i$ ($i = 0, \dotsc, d$) be the linear basis functions associated with the vertices of $\tilde{K}$.
It holds (e.g.,\ K{\v{r}}{\'{\i}}{\v{z}}ek and Lin~\cite{KriLin95} or Lu et al.~\cite[Lemma 1]{LuHuaQiu14})
\[
   {(\tilde \nabla \tilde \phi_i)}^T \tilde \nabla \tilde \phi_i = \frac{1}{\tilde a^2}, \quad i = 0, \dotsc, d,
\]
where $\tilde \nabla$ is the gradient
operator on $\tilde K$ with respect to $\V{\xi}$. (Recall that $\tilde K$ is equilateral so all of its altitudes are
the same.)

Let $\phi_i(\V{x}) = \tilde \phi_i(F_K^{-1}(\V{x}))$, where $F_K^{-1}$ is the inverse mapping of $F_K$.
Since $F_K$ is affine, $\phi_i$ is also a linear basis function on $K$. The altitudes of $K$ in the metric $\M_K$
are related to the gradient of the linear basis functions by
(cf.~\cite[(25) with $\D_K$ being replaced by $\M_K^{-1}$]{LuHuaQiu14})
\[
   {(\nabla \phi_i)}^T \M_K^{-1} \nabla \phi_i  = \frac{1}{a_{i,K,\M}^2} ,
\]
where $\nabla$ stands for the gradient operator  on $K$ with respect to $\V{x}$.
$\phi_i(\V{x}) = \tilde \phi_i(F_K^{-1}(\V{x}))$ and the chain rule give
\[
\nabla \phi_i = {(F_K')}^{-T}\tilde \nabla \tilde \phi_i .
\]

We are now ready to prove \cref{lem-FK-2-1}:
\begin{align*}
\| {(F_K')}^{-1} \M_K^{-1} {(F_K')}^{-T}\|
& = \max_{\V{v} \ne 0} \frac{\V{v}^T {(F_K')}^{-1} \M_K^{-1} {(F_K')}^{-T} \V{v}} {\V{v}^T \V{v}} 
\\
& \ge  \frac{{(\tilde\nabla\tilde\phi_i)}^T {(F_K')}^{-1} \M_K^{-1} {(F_K')}^{-T} \tilde\nabla\tilde\phi_i}
{{(\tilde\nabla\tilde\phi_i)}^T \tilde\nabla\tilde\phi_i} \\
& =  \frac{ {(\nabla\phi_i)}^T \M_K^{-1} \nabla \phi_i} {\tilde{a}^{-2}} = \frac{\tilde{a}^{2}}{a_{i,K,\M}^2} ,
\end{align*}
which implies
\[
\| {(F_K')}^{-1} \M_K^{-1} {(F_K')}^{-T}\| \ge \max_{i} \frac{\tilde{a}^{2}}{a_{i,K,\M}^2}
= \frac{\tilde{a}^{2}}{a_{K,\M}^2} .
\]
Thus, we obtained the left inequality of \cref{lem-FK-2-1}.

On the other hand, $\tilde \nabla\tilde \phi_i$, $i = 1, \dotsc, d$, form a set of $d$ linearly independent vectors.
Thus, we can represent any $\V{v} \in\mathbb{R}^d$ as
\[
\V{v} = \sum_{i=1}^{d} \alpha_i \tilde\nabla\tilde\phi_i .
\]
Then,
\[
\V{v}^T \V{v} = \sum_{i,j=1}^{d} \alpha_i \alpha_j \tilde\nabla\tilde\phi_i^T  \tilde\nabla\tilde\phi_j
\] 
and
\begin{align*}
\V{v}^T {(F_K')}^{-1} \M_K^{-1} {(F_K')}^{-T} \V{v}
& =  \sum_{i,j=1}^{d} \alpha_i \alpha_j \tilde\nabla\tilde\phi_i^T {(F_K')}^{-1} \M_K^{-1} {(F_K')}^{-T}  \tilde\nabla\tilde\phi_j 
\\
& =  \sum_{i,j=1}^{d} \alpha_i \alpha_j \nabla\phi_i^T \M_K^{-1} \nabla \phi_j  \\
& =  \sum_{i,j=1}^{d} \alpha_i \alpha_j {(\M_K^{-\frac 1 2}\nabla\phi_i)}^T  (\M_K^{-\frac 1 2} \nabla \phi_j ) .
\end{align*}
Thus,
\begin{align*}
\V{v}^T {(F_K')}^{-1} \M_K^{-1} {(F_K')}^{-T} \V{v}
& \le  \sum_{i,j=1}^{d} |\alpha_i| \;  |\alpha_j| \frac{1}{a_{i, K, \M} a_{j, K, \M}} 
 \le  \frac{1}{a_{K, \M}^2} \sum_{i,j=1}^{d} |\alpha_i| \;  |\alpha_j| \\
& = \frac{1}{a_{K, \M}^2} {\left(\sum_{i=1}^{d} |\alpha_i|\right)}^2
\le \frac{d}{a_{K, \M}^2} \sum_{i=1}^{d} \alpha_i^2.
\end{align*}
Then,
\begin{align}
\| {(F_K')}^{-1}\M_K^{-1} {(F_K')}^{-T}\|
& \le 
 \frac{d}{a_{K, \M}^2} \max_{\V{\alpha} \ne 0}
\; \frac{\sum_{i=1}^{d} \alpha_i^2  }
{ \sum_{i,j=1}^{d} \alpha_i \alpha_j \tilde \nabla\tilde\phi_i^T  \tilde \nabla\tilde \phi_j } .
\label{lem-FK-2-3}
\end{align}

We now establish a lower bound on the smallest eigenvalue of the matrix
$B = {( \tilde \nabla\tilde\phi_i^T  \tilde \nabla\tilde \phi_j)}_{i,j=1}^{d}$.
Since $\tilde{K}$ is equilateral, it has the same altitude and the same dihedral angle.
This gives
\[
\tilde \nabla\tilde\phi_i^T  \tilde \nabla\tilde \phi_j = \begin{cases}
\frac{1}{\tilde{a}^2},& \quad i = j \\
- \frac{\cos(\tilde{\alpha}) }{\tilde{a}^2} = - \frac{1 }{d\; \tilde{a}^2}, & \quad i \neq j  
\end{cases}
\]
where $\tilde{\alpha}$ is the dihedral angle between the two faces of $\tilde{K}$ not containing the $i^{\text{th}}$
and $j^{\text{th}}$ vertices each. 
Thus, $B$ is a $Z$-matrix. 
Moreover, from $\sum\limits_{j = 0}^{d} \tilde \nabla\tilde\phi_j = 0$,
\[
\sum\limits_{j = 1}^{d} B_{i,j} = \tilde\nabla\tilde\phi_i^T \sum\limits_{j = 1}^{d} \tilde\nabla\tilde\phi_j
=  - \tilde\nabla\tilde\phi_i^T \tilde\nabla\tilde\phi_{0} 
= \frac{\cos(\tilde{\alpha}) }{\tilde{a}^2} 
= \frac{1}{d\; \tilde{a}^2} > 0 .
\]
This implies that $B$ is an $M$-matrix. We have
\[
\lambda_{\min}(B) \ge \min\limits_{i} \sum\limits_{j = 1}^{d} B_{i,j} \ge \frac{1}{d\; \tilde{a}^2} 
\]
and
\[
\sum\limits_{i,j = 1}^{d} \alpha_i \alpha_j  \tilde \nabla\tilde\phi_i^T  \tilde \nabla\tilde \phi_j \ge 
\frac{1}{d\; \tilde{a}^2} \sum_{i=1}^d \alpha_i^2 .
\]
Thus, from \cref{lem-FK-2-3} we get
\[
\| {(F_K')}^{-1}\M_K^{-1} {(F_K')}^{-T}\| \le \frac{d^2 \; \tilde{a}^2}{a_{K, \M}^2} ,
\]
which gives the right inequality of \cref{lem-FK-2-1}.
\end{proof}

\Cref{lem-FK-2} indicates that
\begin{equation}
\| {(F_K')}^{-1}\M_K^{-1} {(F_K')}^{-T}\| \sim a_{K,\M}^{-2}
\end{equation}
if $\tilde{K}$ is chosen to further satisfy $|\tilde K| = \mathcal{O}(1)$.

It is also interesting to obtain a geometric interpretation for
$\| {(F_K')}^T \M_K F_K'\|$. In this case, we do not need to require that $\tilde{K}$ be equilateral.

\begin{lemma}
\label{lem-FK-1}
Let $\tilde{K}$ and $K$ be two arbitrary simplexes in $\mathbb{R}^d$, $F_K: \tilde K \to K$ the affine mapping between them, and $\M_K$ be a constant symmetric and positive definite matrix. Then,
\begin{equation}
\frac{ h_{K,\M}^2}{\tilde h^2} \le \| {(F_K')}^T \M_K F_K'\| \le \frac{ h_{K,\M}^2}{\tilde \rho^2} ,
\label{lem-FK-1-1}
\end{equation}
where $h_{K,\M}$ is the diameter of $K$ in the metric specified by $\M_K$
and $\tilde h$ and $\tilde \rho$ are the diameter and the in-diameter
of $\tilde K$, respectively.
\end{lemma}

\begin{proof}
Consider any two points $\V{\xi}_1, \V{\xi}_2 \in \tilde K$ and the corresponding points  $\V{x}_1, \V{x}_2 \in K$. Then,
\[
(\V{x}_2 - \V{x}_1) = F_K' (\V{\xi}_2 - \V{\xi}_1) .
\]
This gives
\begin{align}
   {(\V{x}_2 - \V{x}_1)}^T \M_K (\V{x}_2 - \V{x}_1) 
      & = {(\V{\xi}_2 - \V{\xi}_1)}^T {(F_K')}^T \M_K F_K'  (\V{\xi}_2 - \V{\xi}_1)
\label{lem-FK-1-2}
\\
& \le \| {(F_K')}^T \M_K F_K' \| \cdot \| \V{\xi}_2 - \V{\xi}_1 \|^2
\notag \\
& \le \tilde{h}^2 \| {(F_K')}^T \M_K F_K' \| .
\notag
\end{align}
Since $\V{x}_1,\V{x}_2 \in K$ are arbitrary,
\[
h_{K,\M}^2 \le \tilde{h}^2\; \| {(F_K')}^T \M_K F_K' \| ,
\]
which gives the left inequality of \cref{lem-FK-1-1}.

Now consider two arbitrary opposing points $\V{\xi}_1$ and $\V{\xi}_2$ on the sphere
of the largest inscribed ball of $\tilde K$ (with
the diameter $\tilde{\rho}$). Dividing both sides of \cref{lem-FK-1-2}
by $\| \V{\xi}_1 - \V{\xi}_2\|^2 =  \tilde{\rho}^2$, we get
\[
\frac{{(\V{x}_2 - \V{x}_1)}^T \M_K (\V{x}_2 - \V{x}_1)}{\tilde{\rho}^2}
= \frac{ {(\V{\xi}_2 - \V{\xi}_1)}^T {(F_K')}^T M_K F_K'  (\V{\xi}_2 - \V{\xi}_1)}{\| \V{\xi}_1 - \V{\xi}_2\|^2 } .
\]
Taking the maximum over all points on the sphere of the largest inscribed ball, the right-hand side is
equal to $\| {(F_K')}^T M_K F_K' \|$ while the left-hand side is less than $h_{K,M}^2/\tilde{\rho}^2$. Hence,
\[
\|{(F_K')}^T M_K F_K' \| \le \frac{h_{K,M}^2}{\tilde{\rho}^2} ,
\]
which gives the right inequality of \cref{lem-FK-1-1}.
\end{proof}

\Cref{lem-FK-1} implies that $\| {(F_K')}^T \M_K F_K'\|$ is equivalent to $h_{K,\M}^2$, i.e.,
\begin{equation}
\| {(F_K')}^T \M_K F_K'\| \sim h_{K,\M}^2,
\end{equation}
when $\tilde{K}$ is chosen to be a unitary equilateral simplex.

Note that interchanging the roles of $K$ and $\tilde{K}$ and replacing $\M_K$ by $\M_K^{-1}$
in \cref{lem-FK-1} provide bounds for $\| {(F_K')}^{-1}\M_K^{-1} {(F_K')}^{-T}\|$ as well.
However, these bounds are not as sharp as bounds in \cref{lem-FK-2}.

\subsection{Mesh nonsingularity}

We first consider the semi-discrete MMPDE \cref{mmpde-3}.
In practical computation, proper modifications of the MMPDE for boundary vertices are required.
Since the analysis is similar for the MMPDE with or without these modifications,
for simplicity in the following we consider only the case without modifications.

\begin{theorem}
\label{thm-ns-1}
Assume that the meshing functional \cref{fun-1} satisfies the coercivity condition
\begin{equation}
   \label{coercive-1}
   G(\J, \det(\J), \M, \V{x}) 
   \ge \alpha 
   {\left [ \tr(\J \M^{-1} \J^T)\right ]}^q-\beta,
   \quad \forall \V{x} \in \Omega,
\end{equation}
with $q > d/2$, where $\alpha > 0$ and $\beta \ge 0$ are constants.

Then, the elements of the mesh trajectory of the semi-discrete MMPDE \cref{mmpde-3} will have positive volumes for $t > 0$ if they have positive volumes initially.

Moreover, their minimum altitudes in the metric $\M$ and their volumes are bounded below by
\begin{align}
   \label{aK-3}
   a_{K,\M} &\ge C_1
   \; \underline{\rho}^{\frac{2 q}{2q-d}}
   \; \overline{m}^{-\frac{d}{2 (2q - d)}}
   \; N^{-\frac{2q}{d (2q-d)}}
   \quad \forall K \in \Th,
   \quad \forall t > 0
   ,
\\
\label{aK-4}
   |K| &\ge C_2
   \; \underline{\rho}^{\frac{2 q d}{2q-d}}
   \; \overline{m}^{-\frac{d^2}{2 (2q - d)}-\frac{d}{2}}
   \; N^{-\frac{2q}{ (2q-d)}} \quad \forall K \in \Th,
   \quad \forall t > 0
   ,
\end{align}
where $C_1$ and $C_2$ are constants given by
\begin{equation}
   C_1 = {\left(
         \frac{\alpha \hat{a}^{2 q} }{d!\; \hat{h}^{2 q} \left (\beta |\Omega| + I_h(\Th(0))\right )}
      \right)}^{\frac{1}{2q-d}},
   \quad C_2 = \frac{C_1^d}{d!} ,
\label{C1C2}
\end{equation}
$\hat{h}$ and $\hat{a}$ are the diameter and height of $\hat{K}$, and
$\overline{m}$ and $\underline{\rho}$ are constants defined in \cref{M-1} and \cref{c-mesh-1}.
\end{theorem}

\begin{proof}
Recall that \cref{mmpde-3} is a gradient system. As a consequence,
\begin{equation}
\frac{d I_{h}}{d t}
= \sum_{i=1}^{N_v} \p{I_{h}}{\V{x}_i} \frac{d \V{x}_i}{d t}
= - \sum_{i=1}^{N_v} \frac{P_i}{\tau} \left \| \p{I_{h}}{\V{x}_i} \right \|^2 \le 0 .
\label{Ih-b1}
\end{equation}
This implies 
\begin{equation}
\label{Ih-b2}
I_h(\Th(t)) \le I_h(\Th(0)),
\end{equation}
where $\Th(t) \equiv (\V{x}_1(t), \dotsc, \V{x}_{N_v}(t))$ is the mesh at time $t$.
From the coercivity \cref{coercive-1}, we get
\begin{equation}
   \label{Ih-b3}
   I_h(\Th(t)) 
   \ge \alpha \sum_{K\in \Th} |K|
      {\left [\tr( (\hat{E}_{K} E_K^{-1}) \M_K^{-1} {(\hat{E}_{K} E_K^{-1})}^T)\right ]}^q
   - \beta |\Omega|
   .
\end{equation}

Denote the edge matrix of $\hat{K}$ by $\hat{E}$ ($\hat{K}$ is the unitary equilateral simplex).
Then,
\begin{align*}
\tr\left( (\hat{E}_{K} E_K^{-1}) \M_K^{-1} {(\hat{E}_{K} E_K^{-1})}^T\right) 
& \ge \| (\hat{E}_{K} E_K^{-1}) \M_K^{-1} {(\hat{E}_{K} E_K^{-1})}^T \| \\
& = \| (\hat{E}_{K} \hat{E}^{-1}) \left ( \hat{E} E_K^{-1} \M_K^{-1}E_K^{-T}\hat{E}^T\right )
   {(\hat{E}_{K} \hat{E}^{-1})}^T \| \\
& \ge  \| \hat{E} E_K^{-1} \M_K^{-1}E_K^{-T}\hat{E}^T \| \cdot
\lambda_{\min}\left ((\hat{E}_{K} \hat{E}^{-1})  {(\hat{E}_{K} \hat{E}^{-1})}^T \right ) \\
& =  \frac{ \| \hat{E} E_K^{-1} \M_K^{-1}E_K^{-T}\hat{E}^T \| }
{\| {(\hat{E}\hat{E}_{K}^{-1} )}^T  (\hat{E}\hat{E}_{K}^{-1}) \|} .
\end{align*}
Applying \cref{lem-FK-1} (with $F_K' := \hat{E}\hat{E}_{K}^{-1}$, $\tilde{K} := K_c$, $K := \hat{K}$,
and $\M_K := I$) and using \cref{c-mesh-1}, we get
\[
\| {(\hat{E}\hat{E}_{K}^{-1} )}^T  (\hat{E}\hat{E}_{K}^{-1}) \| \le \frac{\hat{h}^2}{\rho_{K_c}^2}
\le \frac{\hat{h}^2 N^{\frac 2 d}}{\underline{\rho}^2} .
\]
Thus, we have
\[
\tr( (\hat{E}_{K} E_K^{-1}) \M_K^{-1} {(\hat{E}_{K} E_K^{-1})}^T)  \ge
\frac{\underline{\rho}^2} {\hat{h}^2 N^{\frac 2 d}} \| \hat{E} E_K^{-1} \M_K^{-1}E_K^{-T}\hat{E}^T \|.
\]
Applying \cref{lem-FK-2} (with $F_K' := E_K \hat{E}^{-1}$, $\tilde{K} := \hat{K}$, and $K := K$) gives
\[
   \| \hat{E} E_K^{-1} \M_K^{-1}E_K^{-T}\hat{E}^T \| \ge \frac{\hat{a}^2}{a_{K,\M}^2}
   ,
\]
which leads to
\[
\tr( (\hat{E}_{K} E_K^{-1}) \M_K^{-1} {(\hat{E}_{K} E_K^{-1})}^T)  \ge
\frac{\underline{\rho}^2 \hat{a}^2 } {\hat{h}^2 a_{K,\M}^2 N^{\frac 2 d}} .
\]
Inserting this into \cref{Ih-b3}, we get
\[
\frac{\alpha \underline{\rho}^{2q} \hat{a}^{2q} } {\hat{h}^{2 q}  N^{\frac{2 q}{d}}}
\sum_{K\in \Th} \frac{|K|}{ a_{K,\M}^{2 q} }
- \beta |\Omega| \le I_h(\Th(t)) ,
\]
or, using \cref{Ih-b2},
\begin{equation}
\sum_{K\in \Th} \frac{|K|}{ a_{K,\M}^{2 q} } \le \frac{\hat{h}^{2 q}  N^{\frac{2 q}{d}}} {\alpha \underline{\rho}^{2q} \hat{a}^{2q} }
\left ( \beta |\Omega| + I_h(\Th(0))\right ) .
\label{aK-1}
\end{equation}
Moreover, from \cref{M-1} we have
\[
|K| = \frac{|K| \sqrt{\det(\M_K)}}{\sqrt{\det(\M_K)}}
\ge \frac{a_{K,\M}^d}{d!\, \overline{m}^{\frac{d}{2}}} .
\] 
Combining this with \cref{aK-1} leads to
\begin{equation}
\label{aK-2}
\sum_{K\in \Th} \frac{1}{a_{K,\M}^{2q-d}} \le \frac{d!\, \overline{m}^{\frac{d}{2}} \hat{h}^{2 q}  N^{\frac{2 q}{d}}}
{\alpha \underline{\rho}^{2q} \hat{a}^{2q} }
\left ( \beta |\Omega| + I_h(\Th(0))\right ) ,
\end{equation}
which gives rise to \cref{aK-3} and \cref{aK-4}.

Finally, the volumes of the elements will stay positive if they are positive initially. To show this,
we recall that $G$ is assumed to have continuous derivatives up to the third order for $\|\J\| < \infty$,
$|\det(\J)| < \infty$, and $\V{x} \in \Omega$. Then, $G$ and its derivatives appearing in
\crefrange{mmpde-4}{mmpde-6} are bounded when their arguments
\[
\J := \hat{E}_K E_K^{-1}, \quad \det(\J) := \frac{\det(\hat{E}_K)}{\det(E_K)},\quad
\M := \M_K, \quad \V{x} := \V{x}_K
\]
are bounded. The latter is true since $\Tc$ is given (and fixed), $\M$ satisfies \cref{M-1},
$|\det(E_K)| = d! \, |K|$ is bounded away from zero as shown in \cref{aK-4},
and the vertices stay on $\Omega$ (and their coordinates are bounded).
The other factors in \crefrange{mmpde-4}{mmpde-6} that do not involve $G$
can also be shown to be bounded using the same argument and the fact~\cite[eq.~(38)]{HuaKam15a} that
\[
\begin{bmatrix} \p{\phi_{1,K}}{\V{x}}\\ \vdots \\ \p{\phi_{d,K}}{\V{x}} \end{bmatrix} = E_K^{-1} .
\]
Thus, the nodal mesh velocities are bounded if $|K|$ satisfies \cref{aK-4}.
This bound is global in the sense that it is independent of time and individual elements.
As a consequence, the mesh vertices will move continuously with time and the volumes of the elements cannot
jump over the bound \cref{aK-4} to become negative.
Thus, the volumes of the elements will stay positive and bounded from below if they are positive initially.
\end{proof}

\begin{remark}
\label{rem-ns-0}
From inequality \cref{aK-3} we can see that the ratio of $a_{K,\M}$ to the average element diameter, $N^{-\frac{1}{d}}$,
is bounded below by
\begin{equation}
\frac{a_{K,\M}}{N^{-\frac{1}{d}}} \ge
C_1
\; \underline{\rho}^{\frac{2 q}{2q-d}}
\; \overline{m}^{-\frac{d}{2 (2q - d)}}
\; N^{-\frac{1}{(2q-d)}}
\quad \forall K \in \Th.
\label{aK-5}
\end{equation}
This implies that the larger $q$ is, the closer $a_{K,\M}$ is to the average element diameter.
In particular, when $q \to \infty$, we have $a_{K,\M} \to \mathcal{O}(N^{-\frac{1}{d}})$ and the mesh
is close to being quasi-uniform. Similarly, from \cref{aK-4} we have
\begin{equation}
\frac{|K|}{N^{-1}} \ge
C_2
\; \underline{\rho}^{\frac{2 q d}{2q-d}}
\; \overline{m}^{-\frac{d^2}{2 (2q - d)}-\frac{d}{2}}
\; N^{-\frac{d}{ (2q-d)}}
\quad \forall K \in \Th.
\label{aK-6}
\end{equation}
For example, for Huang's functional \cref{huang} in 2D with $p=1.5$ and $q = pd/2 = 1.5$ we have $|K| \gtrsim N^{-3}$.
Note that this is a rather pessimistic worst case estimate. Recall that the functional \cref{huang} is designed to make the mesh
to satisfy the equidistribution and alignment conditions as closely as possible.
The equidistribution condition takes the form
\[
|K| \sqrt{\det(\M_K)} = \frac{\sigma_h}{N},\quad K \in \Th
\]
where $\sigma_h = \sum_{K} |K| \sqrt{\det(\M_K)}$.
Thus, when a mesh closely satisfies this condition we have
\[
\frac{|\Omega|}{N} {\left (\frac{\ \underline{m}\ }{\ \overline{m}\ }\right )}^{\frac{d}{2}} \le |K|
\le \frac{|\Omega|}{N} {\left (\frac{\ \overline{m}\ }{\ \underline{m}\ }\right )}^{\frac{d}{2}},\quad K \in \Th
\]
which implies $|K| = \mathcal{O}(N^{-1})$. This has been observed in numerical experiment;
e.g., see \cref{ex:sine:wave,fig:sine:wave:KminVsN} in \cref{SEC:numerics}.
\qed{}
\end{remark}

\begin{remark}
\label{rem-ns-1}
The key point of the proof is the energy decreasing property \cref{Ih-b1}. This property
is a crucial advantage of the geometric discretization \cref{fun-3} over 
discretizations based on the continuous MMPDE \cref{mmpde-2}
which, generally speaking, cannot be guaranteed to be a gradient system.
Another key component of the proof is the coercivity assumption \cref{coercive-1}.
Once again, this may not be preserved in general by discretizations based
on the continuous MMPDE \cref{mmpde-2}.
\qed{}
\end{remark}

\begin{remark}
\label{rem-ns-2}
It can be seen  that  Huang's functional \cref{huang} satisfies the coercivity assumption \cref{coercive-1}
with $p > 1$ whereas Winslow's functional does not. In the latter case, we have $q=1$ and \cref{aK-1} still holds.
But \cref{aK-1} with $q=1$ is not sufficient to guarantee a lower bound for $a_{K,\M}$.

It is worth pointing out that the functional of Huang and Russell~\cite[Example 6.2.3]{HuaRus11} also satisfies
the coercivity assumption \cref{coercive-1} for $p > 1$.
\qed{}
\end{remark}

\begin{remark}
\label{rem-ns-3}
The quantity $I_h$ defined in \cref{fun-3} can be viewed as a measure for mesh quality.
The smaller $I_h$ is, the better the mesh quality.
Then, \cref{Ih-b1} implies that the mesh quality improves when $t$ increases.
\qed{}
\end{remark}

\vspace{20pt}

We now consider the time integration of \cref{mmpde-3}. 
Denote the time instants by $t_n, \ n = 0, 1, \dotso$ with the property $t_n \to \infty$ as $n \to \infty$.
We are interested in methods in the form
\begin{equation}
   \label{one-step-1}
   \Th^{n+1} = \Phi^n(\Th^{n}), \quad n = 0, 1, \dotsc,
\end{equation}
for integrating the MMPDE \cref{mmpde-3}.
Methods in the form \cref{one-step-1} do not have to be one-step methods; the integration from $t_n$ to $t_{n+1}$ can be carried out in more than one step.
From the proof of \cref{thm-ns-1}, we have seen that it is important that the discrete functional $I_h$ is monotonically decreasing with the mesh trajectory.
Thus, we assume that the scheme has the property
\begin{equation}
\label{energy-1}
I_h(\Th^{n+1}) \le I_h(\Th^{n}), \quad n = 0, 1, \dots
\end{equation}
This is satisfied by many schemes such as the forward and the backward Euler,
and algebraically stable Runge-Kutta schemes (including Gauss and Radau IIA schemes) under a time-step restriction involving a local Lipschitz bound of the Hessian matrix of $I_h$ (e.g.,~\cite{HaiLub14,StuHum96}).

\begin{theorem}
\label{thm-ns-2}
Assume that the assumptions of \cref{thm-ns-1} are satisfied,
a numerical scheme in the form \cref{one-step-1}
is applied to MMPDE \cref{mmpde-3}
and the resulting mesh sequence ${\{ \Th^n \}}_{n=0}^\infty$
satisfies the property of monotonically decreasing
energy \cref{energy-1}.
If the time step is sufficiently small (but not diminishing) and
the elements of the mesh trajectory have positive volumes initially, they will have positive volumes
for all $t_n > 0$. Moreover, the minimum altitudes in the metric $\M$ and the element volumes are bounded from below
by \cref{aK-3} and \cref{aK-4}.
\end{theorem}

\begin{proof}
The proof of \cref{aK-3,aK-4} for the fully discrete case is similar to that of \cref{thm-ns-1} for the semi-discrete case.
We only need to show that the volumes of the elements will stay positive if the time step is sufficiently small (but not diminishing).
To this end, we recall that $G$ is assumed to have continuous derivatives up to the third order.
As in the last paragraph of the proof of \cref{thm-ns-1}, we can show that when the mesh satisfies \cref{aK-3,aK-4}, the right-hand side (the velocity field) of \cref{mmpde-4} and its gradient and Hessian are bounded by bounds
independent of time and individual elements.
Then, it can be shown that there exists $\Delta t_0 > 0$ (depending only on the above-mentioned bounds
and thus not diminishing for the time being) such that, if $t_{n+1} - t_n \le \Delta t_0$, then $\|\V{x}_j^{n+1}-\V{x}_j^{n}\|$, $j = 1, \dotsc, N_v$, do not exceed a fixed fraction of the minimal altitude
and, in case an implicit scheme is used for \cref{one-step-1}, Newton's (or some other) iteration for the resulting
nonlinear algebraic equations converges.
This guarantees that the elements of the mesh will not become inverted during the current time step.
The argument can be repeated for the next time step since the new mesh satisfies \cref{aK-3,aK-4}, too.
Thus, the volumes of the elements stay positive for $t_n > 0$.
\end{proof}

\subsection{Existence of limiting meshes and minimizers}

We now investigate the convergence of the mesh trajectory as $t \to \infty$.
First, we consider the semi-discrete case \cref{mmpde-3} and then the fully discrete case.

\begin{theorem}
\label{thm-lim-1}
Under the assumptions of \cref{thm-ns-1}, for any nonsingular initial mesh,
the mesh trajectory $\{\Th(t), t > 0\}$ of MMPDE \cref{mmpde-3} has the following properties.
\begin{itemize}
\item[(a)] $I_h(\Th(t))$ has a limit as $t \to \infty$, i.e., 
\begin{equation}
\label{lim-1}
\lim_{t\to \infty} I_h(\Th(t)) = L .
\end{equation}
\item[(b)] The mesh trajectory has limiting meshes, all of which are nonsingular
and satisfy \cref{aK-3} and \cref{aK-4}.
\item[(c)] The limiting meshes are critical points of $I_h$, i.e., they satisfy
\begin{equation}
\label{lim-2}
\p{I_h}{\V{x}_i} = 0, \quad i = 1, \dotsc, N_v.
\end{equation}
\end{itemize}
\end{theorem}

\begin{proof} (a) $I_h(\Th(t))$ has a limit since it is monotone, decreasing as $t \to \infty$ and
bounded from below by $-\beta |\Omega|$.

(b) \cref{thm-ns-1} implies that
the mesh stays nonsingular for $t > 0$ and its vertices remain on $\overline{\Omega}$
(the closure of $\Omega$). The compactness of $\overline{\Omega}$ means
that $\{\Th(t), t > 0\}$ has limits as $t \to \infty$. Obviously, the limiting
meshes satisfy \cref{aK-3} and \cref{aK-4} and thus are nonsingular.

(c) Consider a convergent mesh sequence $\Th(t_k)$, $k = 1, 2, \dotsc$ with the limit $\Th^{*}$.
We will prove that $\Th^{*}$ satisfies \cref{lim-2} using the contradiction argument: assume that $\Th^{*}$ does not
satisfy \cref{lim-2}. Take a small positive number $\epsilon > 0$
and choose a mesh sequence $\tilde{\Th}(t_k) \equiv \Th(t_k+\epsilon)$, $k = 1, 2, \dots$ From the compactness
of $\overline{\Omega}$, we can choose a convergent subsequence from $\{ \tilde{\Th}(t_k)\}$. Without
loss of generality, we pass the notation and consider $\{ \tilde{\Th}(t_k)\}$ as the subsequence with
the limit $\Th^{**}$. From the definition of $\tilde{\Th}(t_k)$ and Taylor's expansion, we have
\begin{align*}
I_h(\Th^{**}) & = \lim_{k \to \infty} I_h(\tilde{\Th}(t_k))
\\
& = \lim_{k \to \infty} I_h( \dotsc, \V{x}_i(t_k) + \epsilon \frac{d \V{x}_i}{d t}(t_k) + \mathcal{O}(\epsilon^2), \dotso)
\\
& = \lim_{k \to \infty} \left ( I_h(\Th(t_k)) + \epsilon \sum_{i=1}^{N_v} \p{I_h}{\V{x}_i}(\Th(t_k)) \frac{d \V{x}_i}{d t}(t_k)
+ \mathcal{O}(\epsilon^2) \right )
\\
& = \lim_{k \to \infty} \left ( I_h(\Th(t_k)) - \epsilon \sum_{i=1}^{N_v} \frac{P_i}{\tau}
\left \| \p{I_h}{\V{x}_i}(\Th(t_k)) \right \|^2
+ \mathcal{O}(\epsilon^2) \right ) .
\end{align*}
Since $I_h$ and its first and second derivatives are bounded under the conditions \cref{aK-3} and \cref{aK-4},
we can choose $\epsilon$ small enough such that the second term in the about equation dominates
the higher order terms. Moreover, the second term is positive since we have assumed that
$\Th^{*}$ does not satisfy \cref{lim-2}. Thus, from the above equation we get
\[
I_h(\Th^{**}) < I_h(\Th^{*}) .
\]
But this contradicts with \cref{lim-1} since it implies that $I_h(\tilde{\Th}(t_k)) - I_h(\Th(t_k)) \to 0$ as $k \to \infty$
or $I_h(\Th^{**}) - I_h(\Th^{*}) = 0$.
\end{proof}

\begin{theorem}
\label{thm-lim-2}
Under the assumptions of \cref{thm-ns-2}, for any nonsingular initial mesh,
the mesh trajectory $\{\Th^n, n = 0, 1, \dotso\}$ of the scheme \cref{one-step-1} applied to
MMPDE \cref{mmpde-3} has the following properties.
\begin{itemize}
\item[(a)] $I_h(\Th^n)$ has a limit as $n \to \infty$, i.e., 
\begin{equation}
\label{lim-3}
\lim_{n\to \infty} I_h(\Th^n) = L .
\end{equation}
\item[(b)] The mesh trajectory has limiting meshes. All of the those limiting meshes are nonsingular
and satisfy \cref{aK-3} and \cref{aK-4}.
\item[(c)] If we further assume that the scheme satisfies a stronger property of monotonically decreasing energy,
\begin{equation}
\label{energy-2}
\begin{cases}
I_h(\Th^{n+1}) \le I_h(\Th^{n}), \quad n = 0, 1, \dotsc,
\\
I_h(\Th^{n+1}) < I_h(\Th^{n}),\quad \text{if $\Th^{n}$ is not a critical point},
\end{cases}
\end{equation}
then the limiting meshes are critical points of $I_h$, i.e., they satisfy \cref{lim-2}.
\end{itemize}
\end{theorem}

\begin{proof}
The proof for (a) and (b) is similar to that of \cref{thm-lim-2}. The proof for (c) is also similar to
that of \cref{thm-lim-2} except that we choose $\tilde{\Th}^{n_k} = 
\Th^{n_k + 1}$, where $\Th^{n_k}$ is a subsequence converging to $\Th^*$. Then (c) can be proved
using \cref{energy-2} and the contradiction argument.
\end{proof}

\Cref{thm-lim-1,thm-lim-2} state that the values of the functional for the mesh trajectory
are convergent as time increases, which can be used as a stopping criterion for the computation. 
In general, however, there is no guarantee that the mesh trajectory converges.
To guarantee the convergence, a stronger descent in the functional value or a stronger
requirement on the meshing functional is needed.
For example, if the time marching scheme satisfies
\begin{equation}
\label{energy-3}
I_h(\Th^{n+1}) \le I_h(\Th^{n}) - \alpha \sum_{i=1}^{N_v} \left \| \p{I_h}{\V{x}_i}(\Th^n) \right \|^2, \quad n = 0, 1, \dotsc,
\end{equation}
for a positive constant $\alpha$, which is a stronger monotonically decreasing energy property than \cref{energy-2},
then we have
\[
\sum\limits_{n=0}^{\infty} \sum_{i=1}^{N_v} \left \| \p{I_h}{\V{x}_i}(\Th^n) \right \|^2 < \infty,
\]
which in turn means that 
\[
\sum_{i=1}^{N_v} \left \| \p{I_h}{\V{x}_i}(\Th^n) \right \| \to 0 \quad \text{as $\quad n \to \infty$}.
\]
Then, we may expect the mesh trajectory $\{\Th^n, n = 0, 1, \dotso\}$ to converge since typically
$(\Th^{n+1} - \Th^n)$ is proportional to the gradient of $I_h$.

On the other hand, a stronger condition can be placed on the meshing functional. In particular,
$\{\Th^n, n = 0, 1, \dotso\}$ is convergent if $I_h$ has a unique critical point.
To explain this, we consider a special example: the functional \cref{huang} with $\theta = \frac{1}{2}$ or the functional
\cref{huang-2}. In this case, we have
\begin{equation}
\label{huang-4}
I_h = \frac{1}{2} \sum_{K \in \Th} |K| \sqrt{\det(\M_K)}
{\left ( \tr (\hat{E}_K E_K^{-1} \M_K^{-1} E_K^{-T}\hat{E}_K^T )\right )}^{\frac{dp}{2}} ,
\quad \frac{dp}{2} \ge 1 .
\end{equation}
We show that $I_h$ is strongly convex about the variables $\V{\xi}_1, \dotsc, \V{\xi}_{N_v}$, for which it is sufficient to show the term ${\left ( \tr (\hat{E}_K E_K^{-1} \M_K^{-1} E_K^{-T}\hat{E}_K^T )\right )}^{\frac{dp}{2}}$
to be convex about $ E \equiv [\V{\xi}_0^K, \dotsc, \V{\xi}_d^K]$ for any element $K$.
Moreover, since
\[
\frac{d}{d \beta} \beta^{\frac{dp}{2}} = \frac{dp}{2}  \beta^{\frac{dp}{2}-1} \ge 0,\quad
\frac{d^2}{d \beta^2} \beta^{\frac{dp}{2}} = \frac{dp}{2}  (\frac{dp}{2}-1) \beta^{\frac{dp}{2}-2} \ge 0,
\]
from~\cite[Lemma 6.2.1]{HuaRus11} it suffices to show that
$\beta \equiv \tr (\hat{E}_K E_K^{-1} \M_K^{-1} E_K^{-T}\hat{E}_K^T )$ is a convex function
about $E$. 

Let
\[
\V{e} = \begin{bmatrix} 1 \\ \vdots \\1 \end{bmatrix} \in \mathbb{R}^d,
\quad
E_\eta = \begin{bmatrix} \V{\eta}_0^K, \dotsc, \V{\eta}_d^K \end{bmatrix} \in \mathbb{R}^{d \times (d+1)} ,
\]
where $E_\eta$ is an arbitrary matrix representing a perturbation of $E$. 
The quadratic form of the Hessian of $\beta$ with respect to $E$ can be expressed as
\[
\tr\left ( \frac{\partial \tr\left (\frac{\partial \beta}{\partial E} E_{\eta}\right ) }{\partial E} E_{\eta} \right ) ,
\]
where we have used the notation of scalar-by-matrix differentiation (cf. \cref{sbmd-1} and \cref{sbmd-2}
and~\cite{HuaKam15a}).
We first compute $\frac{\partial \beta}{\partial E}$. By examining the relation between $E$ and $\hat{E}_K$,
we get
\[
\frac{\partial \beta}{\partial [\V{\xi}_1^K, \dotsc, \V{\xi}_d^K]} = \frac{\partial \beta}{\partial \hat{E}_K},
\quad
\frac{\partial \beta}{\partial \V{\xi}_0^K} = - \V{e}^T \frac{\partial \beta}{\partial \hat{E}_K} ,
\]
which can be combined into
\[
\frac{\partial \beta}{\partial E} = \begin{bmatrix} - \V{e}^T \\ I \end{bmatrix} \frac{\partial \beta}{\partial \hat{E}_K} ,
\]
where $I$ is the $d$-by-$d$ identity matrix. To find $\frac{\partial \beta}{\partial \hat{E}_K}$,
we look at $\hat{E}_K$ as a function of $t$.
Then,
\begin{align*}
\frac{\partial \beta}{\partial t} = \tr \left ( \frac{\partial (\hat{E}_K E_K^{-1} \M_K^{-1} E_K^{-T} \hat{E}_K^T)}
{\partial t} \right )
= \tr \left ( 2  E_K^{-1} \M_K^{-1} E_K^{-T} \hat{E}_K^T \frac{\partial \hat{E}_K}{\partial t} \right ),
\end{align*}
which gives
\[
\frac{\partial \beta}{\partial \hat{E}_K} = 2  E_K^{-1} \M_K^{-1} E_K^{-T} \hat{E}_K^T .
\]
Thus,
\[
\tr\left (\frac{\partial \beta}{\partial E} E_{\eta} \right )= 
\tr\left ( 2 \begin{bmatrix} - \V{e}^T \\ I \end{bmatrix} E_K^{-1} \M_K^{-1} E_K^{-T} \hat{E}_K^T E_{\eta} \right ) .
\]
Repeating the process,
\begin{align*}
\tr\left ( \frac{\partial \tr\left (\frac{\partial \beta}{\partial E} E_{\eta}\right ) }{\partial E} E_{\eta} \right )
& = 2 \tr \left (   E_\eta \begin{bmatrix} - \V{e}^T \\ I \end{bmatrix} E_K^{-1} \M_K^{-1} E_K^{-T}
\begin{bmatrix} - \V{e}\;  I \end{bmatrix} E_\eta^T \right )
\\
& = 2 \left \| E_\eta \begin{bmatrix} - \V{e}^T \\ I \end{bmatrix} E_K^{-1} \M_K^{-\frac 1 2}\right \|_F^2 \ge 0,
\end{align*}
where $\| \cdot \|_F$ is the Frobenius matrix norm.
The equality in the above equation holds if and only if
\begin{equation}
\V{\eta}_0^K = \cdots = \V{\eta}_d^K .
\label{E-eta}
\end{equation}
Thus, the quadratic form of $I_h$ about $\V{\xi}_1, \dotsc, \V{\xi}_{N_v}$ is zero if and only if
the above equality holds for all $K_c \in \Tc$. Since at least one of the boundary vertices
is held fixed and its perturbation must be zero, \cref{E-eta} applies that $E_\eta = 0$
for the element containing the boundary vertex and then other elements, which means that
$I_h$ is strongly convex. As a consequence, $I_h$ has a unique critical point (which is the minimizer)
when $\Omega_c$ is convex.

Notice that the above uniqueness result is for $I_h$ as a function of the computational coordinates.
For the convergence of the mesh trajectory for \cref{mmpde-3} or its discretization,
we need the uniqueness result for $I_h$ as a function of the physical coordinates.
We use the argument of the functional equivalence described in \cref{SEC:fun}.
We first notice that the continuous functional in \cref{huang-2} is the same as the discrete functional
$I_h$ in \cref{huang-4} for the piecewise linear mapping
$\{ F_K: K_c \to K, \ K \in \Th\}$ and the piecewise constant metric tensor
$\{ \M_K, \ K \in \Th\}$. From the functional equivalence, we can
conclude that \emph{$I_h$ has a unique minimizer either as a function of the coordinates of the
physical vertices as long as $\Omega_c$ is convex}.
Then, the mesh trajectory is convergent.

\section{Numerical examples}
\label{SEC:numerics}

To demonstrate the theoretical findings, in particular the decrease of the meshing functional and the lower positive bound of the element volumes, we present numerical results obtained for several examples for mesh adaptation as well as mesh smoothing in two and three dimensions.
Huang's functional \cref{huang} with $p = \frac{3}{2}$ and $\theta = \frac{1}{3}$ is used in the computation.
The computational mesh is taken as the collection of $N$ copies of $N^{-\frac 1 d} \hat{K}$ where $\hat{K}$ is a given unitary equilateral simplex.
The MMPDE \cref{mmpde-4} with $\tau = 1$ (\cref{ex:2d:smoothing,ex:cami1a,ex:nine:spheres}) and $\tau = 0.01$ (\cref{ex:sine:wave}) is integrated using Matlab explicit ODE solver \emph{ode45} for mesh smoothing and implicit ODE solver \emph{ode15s} for mesh adaptation.
\emph{ode45} and \emph{ode15s} typically take multiple steps from $t_n$ to $t_{n+1}$ because they use adaptive step size and they have no options for a single time step.
Boundary vertices are allowed to move along the boundary in all examples but \cref{ex:cami1a}, where they are fixed.  Corner vertices are fixed in all examples.

\begin{example}[2D smoothing] 
\label{ex:2d:smoothing}
We use the MMPDE-based smoothing to improve the mesh quality: we start with an initial mesh,
perturb it (\cref{fig:2d:smoothing:start}) and use $\M = I$ to smooth the perturbed mesh.
\Cref{fig:2d:smoothing:10,fig:2d:smoothing:30} show the resulting mesh at $t=1.0$ and $t=3.0$.
The functional is monotonically decreasing. The minimal element volume is also decreasing
but seems to converge to a positive number and stay bounded from zero. This is consistent
with \cref{thm-ns-2} which states that the element volumes of the mesh is bounded below
by a positive number.

\begin{figure}[p]\centering{}
   \begin{subfigure}[t]{0.31\linewidth}\centering{}
      \includegraphics[clip, width=1.0\linewidth]{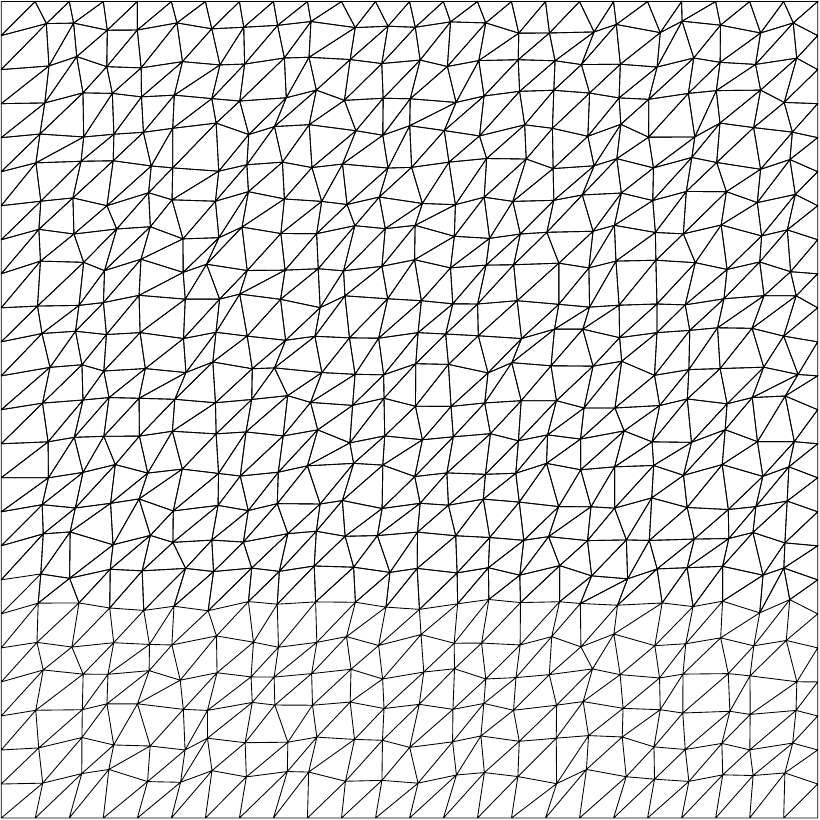}%
      \caption{perturbed initial mesh, $t=0$}\label{fig:2d:smoothing:start}
   \end{subfigure}%
   \hfill%
   \begin{subfigure}[t]{0.31\linewidth}\centering{}
      \includegraphics[clip, width=1.0\linewidth]{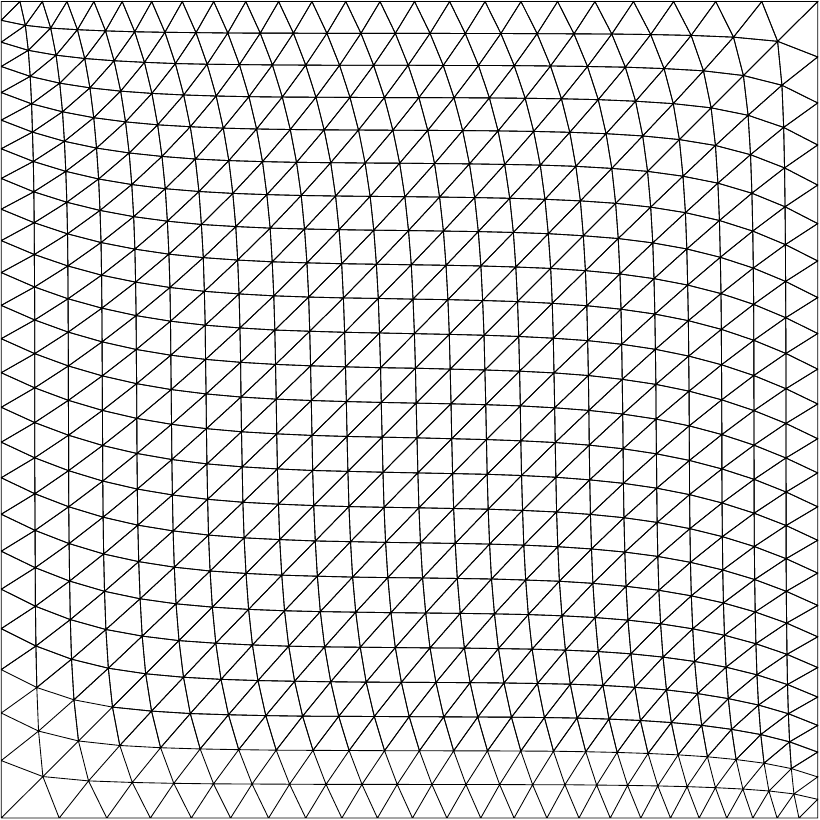}%
      \caption{smoothed mesh at $t=1.0$}\label{fig:2d:smoothing:10}
   \end{subfigure}%
   \hfill%
   \begin{subfigure}[t]{0.31\linewidth}\centering{}
      \includegraphics[clip, width=1.0\linewidth]{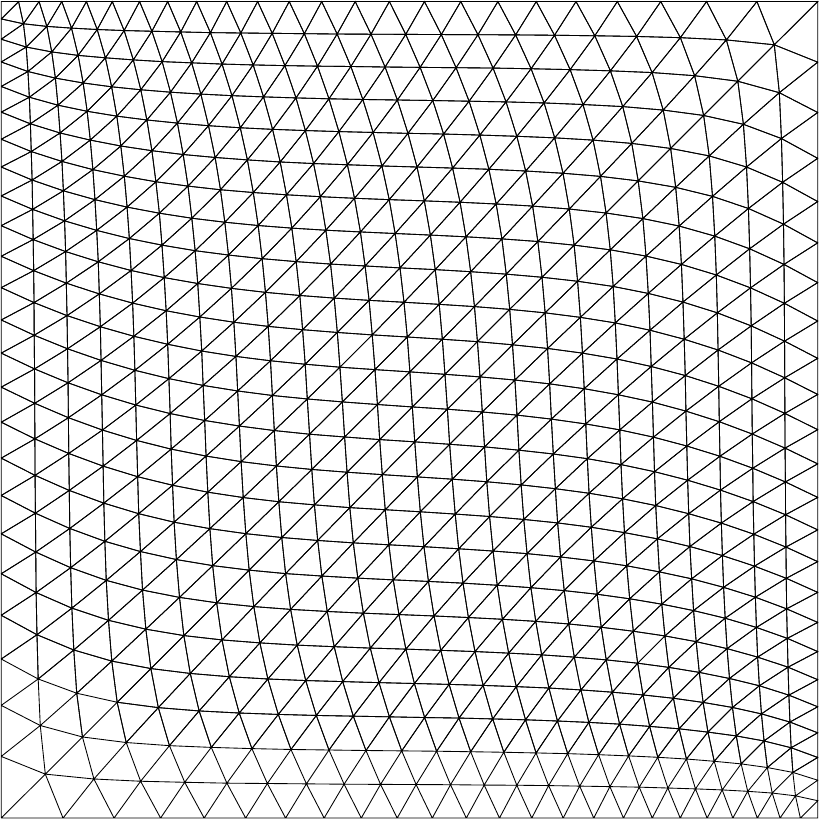}%
      \caption{smoothed mesh at $t=3.0$}\label{fig:2d:smoothing:30}
   \end{subfigure}%
   \\
   \begin{subfigure}[t]{0.40\linewidth}\centering{}
      \begin{tikzpicture}
         \begin{axis}[xlabel={$t$}]
            \addplot[color=blue, solid, mark=\IhMark, line width=1.0pt]
            table [x index=0, y index=1, col sep = space, ]
                  {2d-smoothing.txt};
            \addlegendentry{$I_h$}
         \end{axis}
      \end{tikzpicture}%
      \caption{$I_h$ as function of $t$}
   \end{subfigure}%
   \qquad%
   \begin{subfigure}[t]{0.40\linewidth}\centering{}
      \begin{tikzpicture}
         \begin{axis}[xlabel={$t$}]
            \addplot[color=blue, solid, mark=\KminMark, line width=1.0pt]
               table [x index=0, y index=2, col sep = space]
                  {2d-smoothing.txt};
            \addlegendentry{$\Abs{K}_{\min}$}
         \end{axis}
      \end{tikzpicture}%
      \caption{$\Abs{K}_{\min}$ as function of $t$}
   \end{subfigure}%
   \caption{smoothing of a distorted 2D mesh (\cref{ex:2d:smoothing})}\label{fig:2d:smoothing}%
\end{figure}
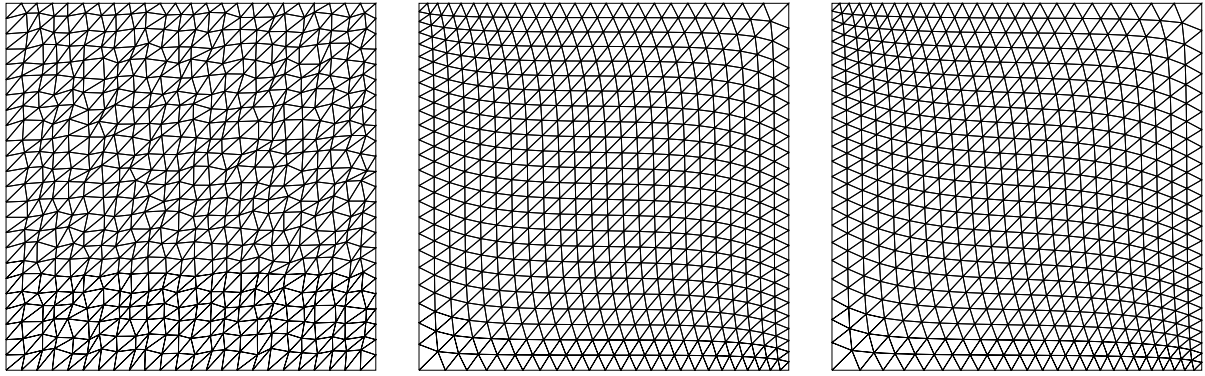
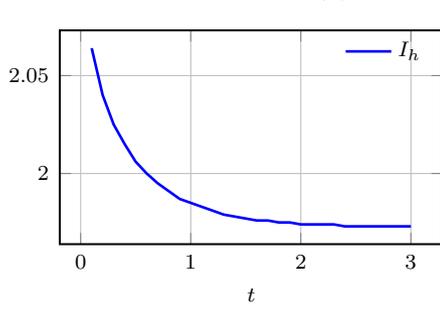
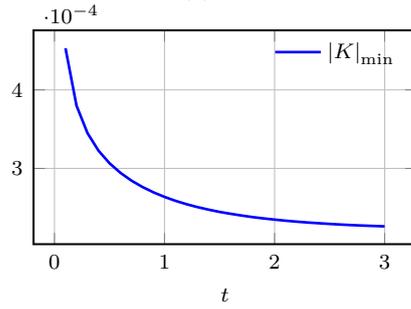

\end{example}

\begin{example}[2D mesh adaptation for the sine wave]
\label{ex:sine:wave}
In this example, the metric tensor $\M$ is based on optimizing the piecewise linear interpolation error measured in the the $L^2$-norm~\cite{Hua05a,HuaSun03},
\begin{equation}
   \M = {\det \left(\alpha I +  \Abs{H(u)} \right)}^{- \frac{1}{6}}
   \left [ \alpha I  +  \Abs{H(u)} \right ] ,
   \label{M-2}
\end{equation}
where $H(u)$ is the recovered Hessian of $u$, $|H(u)|$ is the eigen-decomposition of $H(u)$ with the eigenvalues being replaced by their absolute values, and the regularization parameter $\alpha > 0$ is chosen such that
\begin{equation}
   \int_\Omega \sqrt{\det(\M)} d \V{x}
      = 2 \int_\Omega {\det\left( |H(u)|\right)}^{\frac{1}{3}} \dx
      .
   \label{alpha-1}
\end{equation}
We choose $\Omega = (0,1) \times (0,1)$ and
\[
   u(x, y) 
      = \tanh\left(-20 \left[ y - 0.5 - 0.25 \sin \left(2 \pi  x\right) \right] \right)
   .
\]
\Cref{fig:sine:wave:0,fig:sine:wave:1,fig:sine:wave:3} show the adaptive mesh at $t=0$, $1.0$, and $3.0$ for a $44\times44$ mesh.
As expected, the functional energy is monotonically decreasing (\cref{fig:sine:wave:Ih})
and $\Abs{K}_{\min}$ stays bounded from below (\cref{fig:sine:wave:Kmin}).
Moreover, for a sequence of grids with $N \rightarrow\infty$, it seems that $\Abs{K}_{\min} \sim N^{-1}$ (\cref{fig:sine:wave:KminVsN}), which is in consistent with \cref{aK-6}, which reads as $|K| \ge C \overline{m}^{-3} N^{-3}$ for this example.

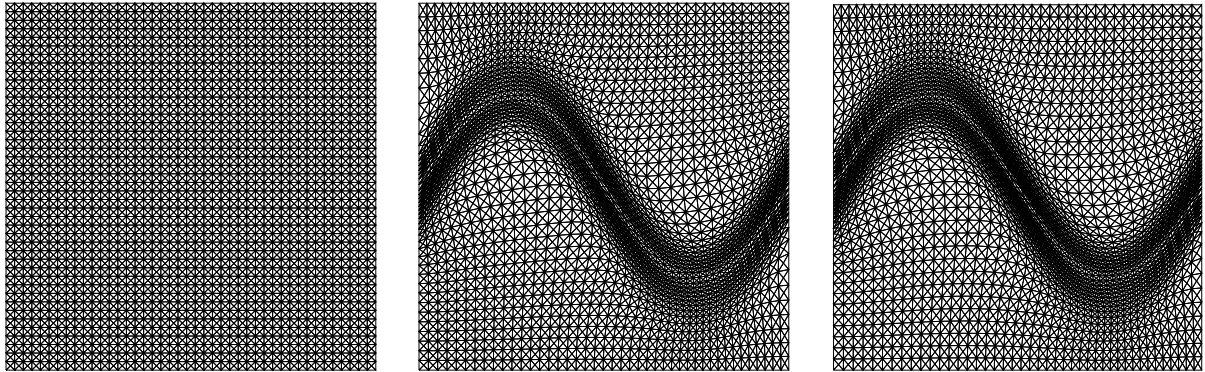
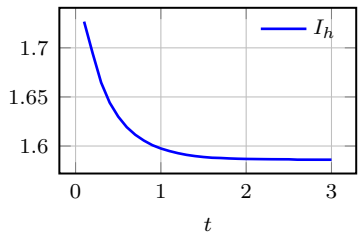
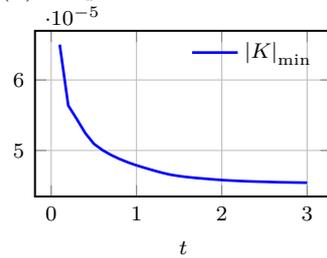
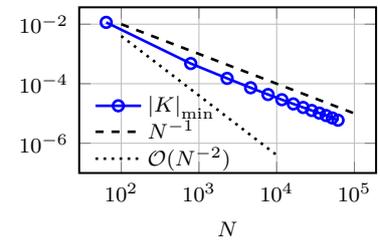
\begin{figure}[p]\centering{}
   \begin{subfigure}[t]{0.31\linewidth}\centering{}
      \includegraphics[clip, width=1.0\linewidth]{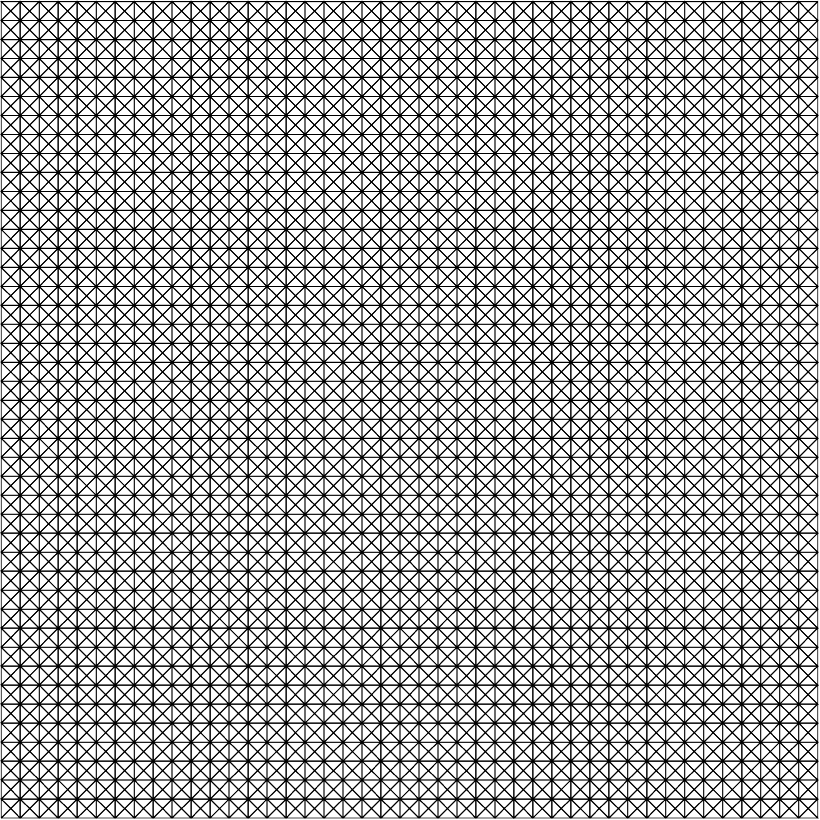}%
      \caption{initial mesh at $t=0.0$}\label{fig:sine:wave:0}
   \end{subfigure}%
   \hfill%
   \begin{subfigure}[t]{0.31\linewidth}\centering{}
      \includegraphics[clip, width=1.0\linewidth]{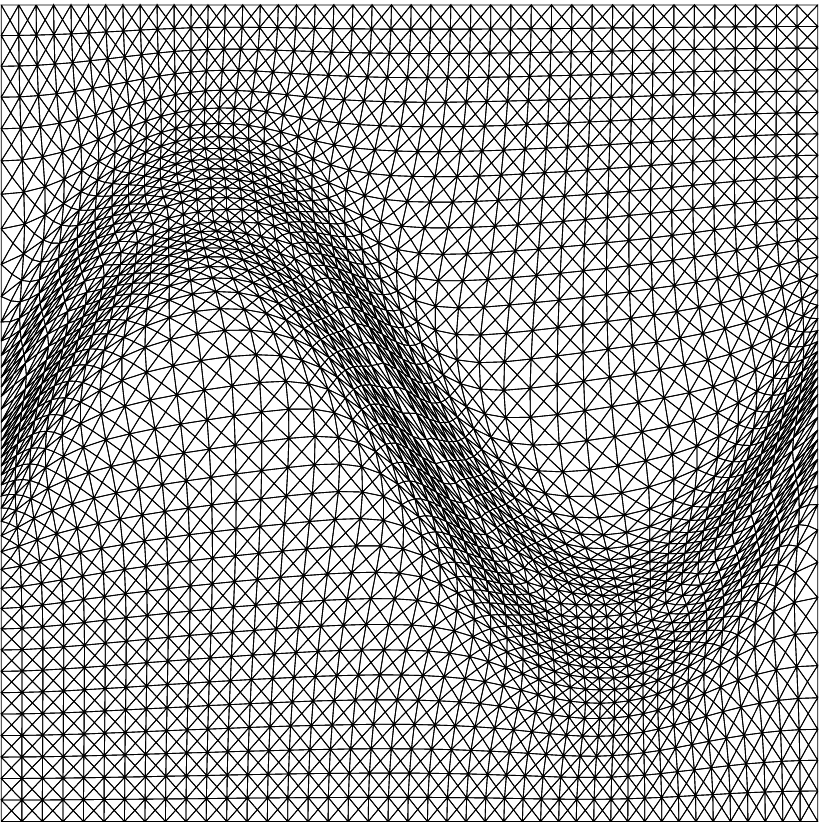}%
      \caption{adaptive mesh at $t=1.0$}\label{fig:sine:wave:1}
   \end{subfigure}%
   \hfill%
   \begin{subfigure}[t]{0.31\linewidth}\centering{}
      \includegraphics[clip, width=1.0\linewidth]{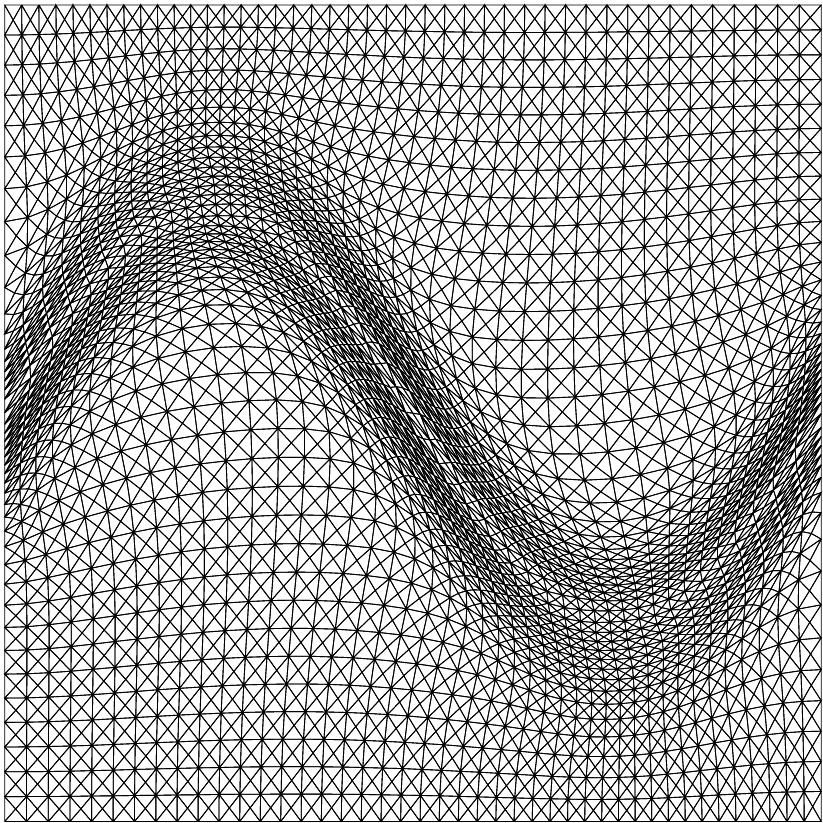}%
      \caption{adaptive mesh at $t=3.0$}\label{fig:sine:wave:3}
   \end{subfigure}%
   \\ 
   \begin{subfigure}[t]{0.31\linewidth}
      \begin{tikzpicture}\centering{}
         \begin{axis}[xlabel={$t$}]
            \addplot[color=blue, solid, mark=\IhMark, line width=1.0pt]
            table [x index=0, y index=1, col sep = space, ]
                  {2d-data.txt};
            \addlegendentry{$I_h$}
         \end{axis}
      \end{tikzpicture}%
      \caption{$I_h$ as function of $t$}\label{fig:sine:wave:Ih}
   \end{subfigure}%
   \hfill%
   \begin{subfigure}[t]{0.31\linewidth}\centering{}
      \begin{tikzpicture}
         \begin{axis}[xlabel={$t$}]
            \addplot[color=blue, solid, mark=\KminMark, line width=1.0pt]
               table [x index=0, y index=2, col sep = space]
                  {2d-data.txt};
            \addlegendentry{$\Abs{K}_{\min}$}
         \end{axis}
      \end{tikzpicture}%
      \caption{$\Abs{K}_{\min}$ as function of $t$}\label{fig:sine:wave:Kmin}
   \end{subfigure}%
   \hfill%
   \begin{subfigure}[t]{0.31\linewidth}\centering{}
      \begin{tikzpicture}
         \begin{loglogaxis}[xlabel={$N$}, ymin = 1.0e-07, legend columns=1,
            legend style={at={(0.02,-0.05)}, anchor=south west},
            xmax = 1.9e+05]
            \addplot[color=blue, solid, mark=\KminConvMark, line width=1.0pt,]
               table [x index=0, y index=1, col sep = space,]
                  {2d-eq35.txt};
            \addlegendentry{$\Abs{K}_{\min}$}
            \addplot[color=black, dashed, mark=none, line width=1.0pt,
               table/row sep=\\,]
               table {1.0e+02 1.0e-02\\ 1.0e+05 1.0e-05\\};
            \addlegendentry{$N^{-1}$}
            \addplot[color=black, dotted, mark=none, line width=1.0pt,
               table/row sep=\\,]
               table {1.0e+02 4.0e-03\\ 1.0e+04 4.0e-07\\};
            \addlegendentry{$\mathcal{O}(N^{-2})$}
         \end{loglogaxis}
      \end{tikzpicture}%
      \caption{$\Abs{K}_{\min}$ as $N\rightarrow\infty$}\label{fig:sine:wave:KminVsN}
   \end{subfigure}%
   \caption{2D mesh adaptation for the sine wave (\cref{ex:sine:wave})}\label{fig:sine:wave}%
\end{figure}

\end{example}

\begin{example}[3D smoothing, cami1a]
\label{ex:cami1a}
This example demonstrates smoothing  of a tetrahedral mesh generated by \emph{TetGen}~\cite{Si15} for the \emph{cami1a} geometry (\cref{fig:cami1a:part}).
For this example too, the functional is monotonically decreasing (\cref{fig:cam:ih}) and $\Abs{K}_{\min}$ stays bounded from below (\cref{fig:cam:kmin}).
The dihedral angle statistics of the original \emph{TetGen} mesh with
those after mesh smoothing (\cref{tab:cami1a}) shows that smoothing significantly reduces the number of small (\ang{0}--\ang{20}) and large (\ang{150}--\ang{180}) dihedral angles and, thus, produce a more uniform mesh.

\begin{figure}[t]\centering{}
   \begin{subfigure}[t]{0.26\linewidth}\centering{}
      \includegraphics[clip, width=1.0\linewidth]{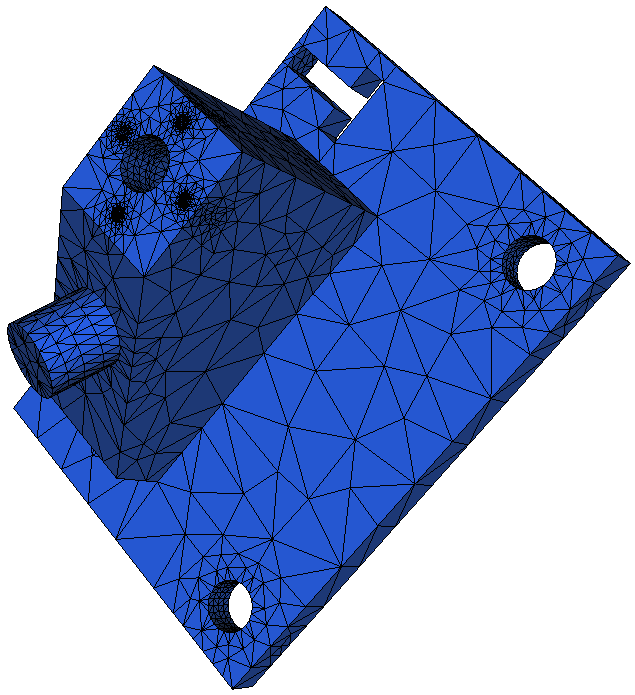}%
      \caption{mesh geometry}\label{fig:cami1a:part}
   \end{subfigure}%
   \hfill{}%
   \begin{subfigure}[t]{0.35\linewidth}
      \begin{tikzpicture}
         \begin{axis}[xlabel={$t$}]
            \addplot[color=blue, solid, mark=\IhMark, line width=1.0pt]
            table [x index=0, y index=1, col sep = space, ]
                  {3d-smoothing.txt};
            \addlegendentry{$I_h$}
         \end{axis}
      \end{tikzpicture}%
      \caption{$I_h$ as function of $t$\label{fig:cam:ih}}
   \end{subfigure}%
   \hfill{}%
   \begin{subfigure}[t]{0.35\linewidth}\centering{}
      \begin{tikzpicture}
         \begin{axis}[xlabel={$t$}]
            \addplot[color=blue, solid, mark=\KminMark, line width=1.0pt]
               table [x index=0, y index=2, col sep = space]
                  {3d-smoothing.txt};
            \addlegendentry{$\Abs{K}_{\min}$}
         \end{axis}
      \end{tikzpicture}%
      \caption{$\Abs{K}_{\min}$ as function of $t$\label{fig:cam:kmin}}%
   \end{subfigure}%
   \\
   \begin{subtable}[t]{1.0\linewidth}\centering{}%
      \vspace{0.5ex}
      \caption{statistics of dihedral angles
         before and after smoothing}\label{tab:cami1a}
         \vspace{1ex}
      \begin{tabular}[b]{@{} crr|crr @{}}
         \toprule%
         angle & before & after & angle & before  & after \\
         \midrule%
            \input{cami1a-dihedral.txt}
         \bottomrule%
      \end{tabular}%
   \end{subtable}
   \caption{Smoothing of a 3D cami1a mesh with \num{18980} elements (\cref{ex:cami1a})}
\end{figure}
\end{example}

\begin{example}[3D mesh adaptation for nine spheres]
\label{ex:nine:spheres}
In this example we choose $\Omega = (-1,1) \times (-1,1) \times (-1,1)$
and $\M$ from \cref{M-2} to minimize the $L^2$ interpolation error bound for
\begin{align*}
   u(x, y, z) &=
      \tanh\left(30\left[{(x-0.0)}^2 + {(y-0.0)}^2 + {(z-0.0)}^2 - 0.1875\right]\right)\\
   &+ \tanh\left(30\left[{(x-0.5)}^2 + {(y-0.5)}^2 + {(z-0.5)}^2 - 0.1875\right]\right)\\
   &+ \tanh\left(30\left[{(x-0.5)}^2 + {(y+0.5)}^2 + {(z-0.5)}^2 - 0.1875\right]\right)\\
   &+ \tanh\left(30\left[{(x+0.5)}^2 + {(y-0.5)}^2 + {(z-0.5)}^2 - 0.1875\right]\right)\\
   &+ \tanh\left(30\left[{(x+0.5)}^2 + {(y+0.5)}^2 + {(z-0.5)}^2 - 0.1875\right]\right)\\
   &+ \tanh\left(30\left[{(x-0.5)}^2 + {(y-0.5)}^2 + {(z+0.5)}^2 - 0.1875\right]\right)\\
   &+ \tanh\left(30\left[{(x-0.5)}^2 + {(y+0.5)}^2 + {(z+0.5)}^2 - 0.1875\right]\right)\\
   &+ \tanh\left(30\left[{(x+0.5)}^2 + {(y-0.5)}^2 + {(z+0.5)}^2 - 0.1875\right]\right)\\
   &+ \tanh\left(30\left[{(x+0.5)}^2 + {(y+0.5)}^2 + {(z+0.5)}^2 - 0.1875\right]\right)
   .
\end{align*}

\cref{fig:nine:spheres} shows an example of the final adaptive mesh and plots of the functional value
and $\Abs{K}_{\min}$ with respect to the time.
As expected, the functional value is monotonically decreasing.  $\Abs{K}_{\min}$ is decreasing with time as well but one observes that it is bounded from below. 

\begin{figure}[t]\centering{}
      \begin{subfigure}[t]{0.28\linewidth}\centering{}
      \includegraphics[clip, width=1.0\linewidth]{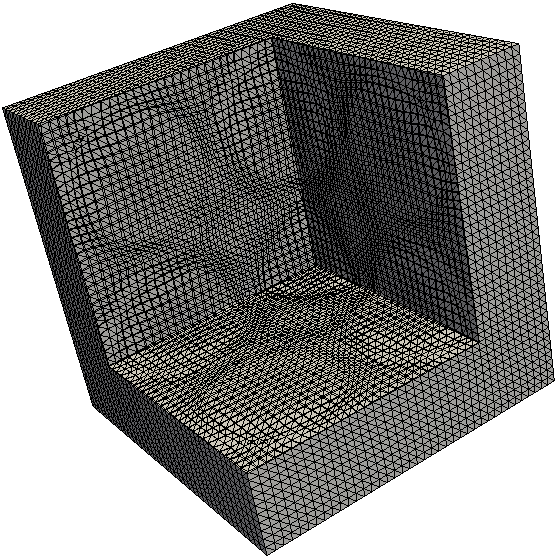}%
      \caption{mesh clip}
   \end{subfigure}%
   \hfill{}%
   \begin{subfigure}[t]{0.35\linewidth}\centering{}
      \begin{tikzpicture}
         \begin{axis}[xlabel={$t$}]
            \addplot[color=blue, solid, mark=\IhMark, line width=1.0pt]
            table [x index=0, y index=1, col sep = space, ]
                  {3d-data.txt};
            \addlegendentry{$I_h$}
         \end{axis}
      \end{tikzpicture}%
      \caption{$I_h$ as function of $t$\label{fig:nine:ih}}
   \end{subfigure}%
   \hfill{}%
   \begin{subfigure}[t]{0.35\linewidth}\centering{}
      \begin{tikzpicture}
         \begin{axis}[xlabel={$t$}]
            \addplot[color=blue, solid, mark=\KminMark, line width=1.0pt]
               table [x index=0, y index=2, col sep = space]
                  {3d-data.txt};
            \addlegendentry{$\Abs{K}_{\min}$}
         \end{axis}
      \end{tikzpicture}%
      \caption{$\Abs{K}_{\min}$ as function of $t$\label{fig:nine:kmin}}
   \end{subfigure}%
   \caption{%
      \cref{ex:nine:spheres} (3D mesh adaptation for nine spheres)
   }\label{fig:nine:spheres}%
\end{figure}
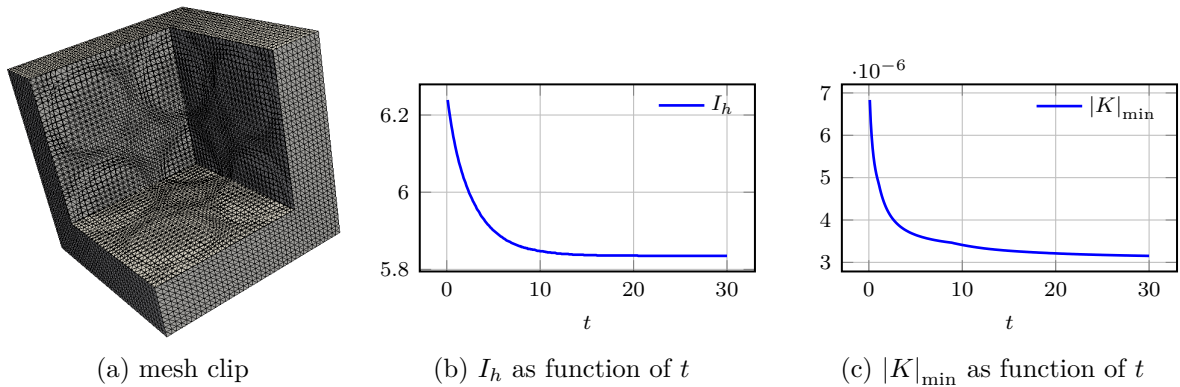
\end{example}

\section{Conclusions}
\label{SEC:conclusion}

The geometric discretization of meshing functionals recently introduced in~\cite{HuaKam15a} can be formulated as a modified gradient system of the corresponding discrete functionals for the location of mesh vertices.

For the semi-discrete system \cref{mmpde-3} and meshing functionals satisfying the coercivity condition \cref{coercive-1} with $q > d/2$ (such as Huang's functional \cref{huang} with $p > 1$), the value of the meshing functional is always convergent and the mesh trajectory has nonsingular limiting meshes.

In particular, \cref{thm-ns-1} shows that the mesh stays nonsingular for $t>0$ if it is nonsingular initially: the altitudes and the volumes of mesh elements stay bounded below by positive numbers depending only on the number of elements, the metric tensor, and the initial mesh, cf.~\cref{aK-3,aK-4}.
Moreover, \cref{thm-lim-1} shows that all limiting meshes are critical points of the discrete functional and satisfy \cref{aK-3,aK-4}.
The convergence of the mesh trajectory can be guaranteed if a stronger condition is placed on the meshing functional.

\Cref{thm-ns-2,thm-lim-2} show that the above-mentioned properties also hold for the fully discrete systems of MMPDE \cref{mmpde-3} provided that the time step is sufficiently small and the underlying integration scheme satisfies the property of monotonically decreasing energy.
For example, Euler, backward Euler, and algebraically stable Runge-Kutta schemes satisfy this property under a mild restriction on the time step.

We would like to point out that the results of the current work cannot be applied directly to non-simplicial meshes.
Nevertheless, a polygonal/polyhedral mesh can first be triangulated into a simplicial mesh (for which the current results can be applied) and then the updated position of the vertices of the original mesh can be obtained through the simplicial mesh.
The interested reader is referred to~\cite{HuaWan} for the application of this idea to polygonal meshes.

\vspace{2ex}
\textbf{Acknowledgments.}
The authors would like to thank the anonymous referee for the valuable comments in improving the quality of the paper.

L.~K.\ is thankful to School of Mathematical Sciences of Xiamen University for the hospitality during his research visit in 2015 to Xiamen, where part of this work has been carried out.

\bibliographystyle{plain}

\end{document}